%% file: Shadows.tex
\documentclass[11pt]{article}
\usepackage[a4paper, margin=1in]{geometry}
\usepackage{amsmath}
\usepackage{amsthm}
\usepackage{bbm}
\usepackage{amssymb}
\usepackage{graphicx}
\graphicspath{{Pictures/}}
\usepackage{amsfonts}
\usepackage{xcolor}
\usepackage{pdfpages}
\usepackage{enumerate}

\usepackage[hidelinks]{hyperref}
\usepackage[justification=centering]{caption}
\usepackage{tikz}
\usepackage{booktabs}

\begin{document}

\newtheorem{theorem}{Theorem}
\numberwithin{theorem}{section}
\newtheorem{proposition}[theorem]{Proposition}
\newtheorem{lemma}[theorem]{Lemma}
\newtheorem{corollary}[theorem]{Corollary}
\theoremstyle{definition}
\newtheorem{definition}[theorem]{Definition}
\newtheorem{example}[theorem]{Example}
\newtheorem{remark}[theorem]{Remark}
\newtheorem{conjecture}[theorem]{Conjecture}
\newtheorem{algorithm}[theorem]{Algorithm}

\newcommand{\uw}{\mathcal{U}(W,X)}
\newcommand{\W}{$(W,S)$}
\newcommand{\ix}{\textbf}
\newcommand{\tr}{\textcolor{red}}
\newcommand{\sg}{$\Sigma$}
\newcommand{\x}{\mathcal{X}}
\newcommand{\C}{\mathcal{C}}
\newcommand{\tw}{T_{w^{-1}}^{-1}}
\newcommand{\mch}{\mathcal{H}}
\newcommand{\aw}{\tilde{A}_2}
\newcommand{\Sh}[1]{\textnormal{Sh}(#1)}
\newcommand{\nSh}[1]{\textnormal{Sh}^\textnormal{c}(#1)}
\newcommand{\PSh}[1]{\textnormal{PSh}(#1)}
\newcommand{\nPSh}[1]{\textnormal{Annex}(#1)}
\newcommand{\wk}{w_k}
\newcommand{\wj}{w_j}
\newcommand{\wi}{w_i}
\newcommand{\B}[1]{B(#1)}
\newcommand{\Hid}{H^{\mathbf{1}}}
\newcommand{\Hinf}{H^\infty}
\newcommand{\cinf}{\mathcal{C}_\infty}

\title{Computations of $\aw$ Bruhat intervals via shadows}

\author{Megan Masters}
\maketitle

\begin{abstract}
    We develop an explicit geometric and algebraic description of shadows in the affine Coxeter complex of type $\aw$, by introducing a coordinate system based on a decomposition of the complex into tunnels and channels. We also provide an algorithm for converting arbitrary reduced words into coordinates. Using this framework, we identify geometric symmetries of shadows and show that they are governed by the underlying channel structure. This allows us to derive explicit, piecewise formulas for the cardinality of shadows in $\aw$, depending on the parity of the coordinates. Furthermore, we establish a simple criterion for shadow membership via a counting function that detects admissible positions within channels. These results provide a concrete and computationally effective description of shadows in affine type $\aw$, bridging the gap between combinatorial definitions and geometric realisations.

\end{abstract}

\section{Introduction}

First introduced in the 1930s, the study of Coxeter groups and their associated geometric realisations, known as Coxeter complexes, remains a central theme in modern algebraic combinatorics and representation theory \cite{COXETER,ABRAMENKOBROWN,RON}. These complexes form fundamental examples of buildings \cite{TITS1}, which were constructed to describe semisimple algebraic groups and Lie groups \cite{TITS}. The combinatorial structure of these groups is deeply encoded in the Bruhat order, a partial ordering that governs the decomposition of flag varieties and the behaviour of Kazhdan--Lusztig polynomials  \cite{KL}. 

Galleries and shadows provide a geometric way to study Bruhat order. For an element $w$ in a Coxeter group $W$, the shadow $\Sh{w}$ consists of the set of all elements $x \in W$ that can be reached via folding operations on a minimal gallery from the identity to $w$. Under the trivial orientation of the Coxeter complex, the shadow $\Sh{w}$ is precisely the Bruhat interval $[1, w]$ \cite{SHA}. While the definition of a shadow is conceptually straightforward, the explicit characterisation of its elements in the infinite affine case presents significant geometric challenges \cite{WILD}. Determining whether one alcove lies within the shadow of another alcove traditionally requires the construction of folded galleries or recursive definitions, processes that lack a direct, closed-form computational approach.

The primary objective of this paper is to give such a direct description in the $\aw$ complex. Among the affine types, the $\aw$ Coxeter complex is the simplest irreducible affine Coxeter complex of rank greater than one and admits a geometric realisation as a tiling of the Euclidean plane by equilateral triangles \cite{DAVIS}. In this paper, we develop an explicit linear-time algorithm in $\aw$ that allows for the precise calculation of shadow membership and size.

The foundation of our approach is the introduction of a new coordinate system on $\aw$. By partitioning the $\aw$ tiling into \ix{tunnels}, sequences of alcoves bounded by parallel hyperplanes (see Definition \ref{tunnelsdef}), we provide unique coordinates for every element in the group. We develop the Words-to-Coordinates Algorithm (Algorthim \ref{algorithm}) that maps any reduced word in the generators $\{s_0, s_1, s_2\}$ to these coordinates and prove that this mapping is invariant under the braid relations $(s_i s_j)^3 = 1$. This leads to the following theorem, which ensures that our algorithm is well defined with respect to the underlying Coxeter group structure.

\begin{theorem}
    For any element $w\in \aw$, and any reduced word $s_{i_1}\cdots s_{i_k}$ representing $w$, Words-to-Coordinates Algorithm outputs the coordinates $(n,m)$ of the alcove corresponding to $w$. 
\end{theorem}

The main geometric objects in our analysis are \ix{channels}, which are introduced in Section \ref{defchannels}. These are infinite sets of alcoves running parallel to one of the three root directions. We show that these channels organise the structure of shadows. We demonstrate that shadows in $\aw$ possess inherent symmetries, particularly across the hyperplanes defining the identity alcove. These symmetries significantly reduce the complexity of shadow calculations. 

The culmination of this work is a series of explicit formulas for the size of shadows. We prove that, for an alcove $(n, m)$, the number of elements in its shadow, denoted $S_{n,m}$, can be calculated directly from its coordinates. For example, when both $n$ and $m$ are even and $n,m \geq 2$, we establish that:
\[
S_{n,m} = \frac{12nm + 3n^2 + 3m^2 - 6n - 6m}{4}.
\]

Similar piecewise relations are derived for all parities of $n$ and $m$. Furthermore, we provide a criterion for shadow membership based on a simple linear function $k(n, m, c)$, which determines the number of elements in the shadow within the $(\alpha+\beta)$-channel containing $(0,c)$. We can then conclude the following theorem.

\begin{theorem}
The alcove $(i,j)$ is in the shadow of $(n,m)$ if and only if 

    \[2-k(n,m,({i+2j})/{2})\leq i \leq k(n,m,({i+2j})/{2}).\]
    This allows us to determine whether any two reduced words $x,y$ satisfy $x\leq y$ in $O(\ell(y))$ time. 
\end{theorem}

The paper is organised as follows. In Section \ref{background}, we recall the necessary background on the $\aw$ Coxeter complex, galleries, and the definition of shadows. In Section \ref{exploring}, we define the $(n, m)$-coordinate system. In Section \ref{algo}, we present and prove an algorithm that takes reduced words and outputs the corresponding coordinates. Then, in Section \ref{sect5}, we explore the geometry of $\aw$ with respect to this coordinate system, and establish the geometric properties of channels. Finally, in Section \ref{calculations}, we present our main results: the boundary characterisation of shadows as convex hulls and the closed-form descriptions of shadows in type $\aw$.

\section{Background}\label{background}
We begin by recalling the definitions and constructions related to the $\aw$ Coxeter complex. Our definitions and statements are taken from \cite{MASTERS}. See this paper for a detailed discussion of the relevant background material, including the link between shadows and Bruhat intervals. 

\subsection{The Coxeter complex}\label{roots}

We take the Coxeter group 
\[\aw=\langle s_0,s_1,s_2\mid s_i^2=1, (s_0s_1)^3=(s_0s_2)^3=(s_1s_2)^3=1\rangle,\]
and construct its \ix{Coxeter complex}. This is a simplicial complex whose maximal simplices are equilateral triangles and are in bijection with elements of the group. These are called \ix{alcoves}. If $x,y\in \aw$ are such that $x=ys_i$ for some $i\in\{0,1,2\}$, then their corresponding simplices are adjacent. The intersection of the simplices corresponding to $x$ and $y$ is a line segment, called a \ix{panel}, of type $i$. This creates a tiling of the plane by equilateral triangles. 

We recall that our Coxeter group acts naturally on this complex. Using the bijection between alcoves and group elements, we define the action of an element $w\in \aw$ on an alcove representing the element $x$ by sending this alcove to the alcove representing the element $wx$. 

\begin{remark}\label{typepreserved}
This map is an isometry, so in particular adjacency is preserved. It also preserves the type of adjacency, and therefore the types of the panels.
\end{remark}

Within this Coxeter complex, we have reflections and walls. These are in bijection with the hyperplanes of the complex. 

\begin{definition}\label{walls}
    An element $r\in \aw$ is called a \ix{reflection} if it is a conjugate of an element $s_i\in S$. The \ix{wall} $H_r$ of a reflection $r$ is the set of simplices in the Coxeter complex that are fixed by $r$ when $r$ acts on the complex by left multiplication. Then $H_r$ is a subcomplex of codimension 1.
\end{definition}

\begin{proposition}\textnormal{\cite[Proposition 2.6]{RON}}
    There is a bijection between the set of reflections of a Coxeter group, and the set of walls in the corresponding Coxeter complex.
\end{proposition}

Each affine Coxeter complex has an underlying root system. For $\aw$, we consider the root system of type $A_2$:

\[E=\{(v_1,v_2,v_3)\in\mathbb{R}^3\mid v_1+v_2+v_3=0\},\]
\[\Phi = \{(1,-1,0),(-1,1,0),(1,0,-1),(-1,0,1),(0,1,-1),(0,-1,1)\}.\]

\begin{definition}
    Let $\langle\cdot,\cdot\rangle$ be the standard inner product on $E$. 
    For $\alpha\in E\backslash\{0\}$, define
    \[\alpha^\vee=\dfrac{2\alpha}{\langle\alpha,\alpha\rangle}.\]
    Define also the set
    \[H_\alpha = \{x\in E\mid \langle x,\alpha\rangle=0\}.\]
    The \ix{orthogonal reflection} in $H_\alpha$ is the map $s_\alpha :E\to E$ with
    \[s_\alpha(x)=x-\langle x,\alpha\rangle\alpha^\vee.\]
\end{definition}

We can now construct our finite and affine Weyl groups connected to this root system. These are precisely the Coxeter groups of types $A_2$ and $\aw$, respectively.
\begin{definition}
    The \ix{(finite) Weyl group} $A_2$ of the root system $\Phi$ is the group generated by the reflections $s_\alpha$ for $\alpha\in \Phi$. 
\end{definition}

We can now define our affine reflections. For each $k\in\mathbb{Z}$, we fix the reflection $s_{\alpha;k}:E\to E$ as
\[s_{\alpha;k}(x)=x-(\langle x,\alpha\rangle-k)\alpha^\vee.\]
For each root $\gamma$ in our root system, and for each $k\in\mathbb{Z}$, we have the hyperplane
\[H_{\gamma,k}=\{x\in E\mid \langle x, \gamma\rangle = k\}.\]

\begin{definition}
    The \ix{affine Weyl group} $\aw$ of the root system $\Phi$ is the group generated by the reflections $s_{\alpha;k}$ for $\alpha\in \Phi$, $k\in \mathbb{Z}$. 
\end{definition}

\subsection{Galleries}

We now introduce galleries in the Coxeter complex and recall the folding operations used to define shadows.

\begin{definition}\label{comb.gallery}
    Given the Coxeter complex of $\aw$, a \ix{combinatorial gallery} is a sequence
    \[\gamma = (c_0,p_1,c_1,p_2,\hdots  ,p_n,c_n),\]
    where the $c_i$ are alcoves and the $p_i$ are panels of $\aw$, such that $p_i$ is contained in $c_{i-1}$ and $c_{i}$ for all $i=1,\hdots ,n$. The length of $\gamma$ is $n$. The gallery $\gamma$ is \ix{minimal} if there does not exist a shorter gallery starting at $c_0$ and ending at $c_n$. 
\end{definition}

\begin{definition}\label{gallerytype}
    Let $\gamma=(c_0,p_1,c_1,p_2,\hdots  ,p_n,c_n)$ be a combinatorial gallery. The \ix{type} of $\gamma$ is the word $f=j_1\hdots j_n$ in $\{0,1,2\}$, where panel $p_i$ has type $j_i$. 
\end{definition}

\begin{definition}
    Consider a gallery $\gamma = (c_0,p_1,c_1,\hdots ,p_n,c_n)$. Let panel $p_i$ of $\gamma$ have type ${j_i}\in I$. We define its \ix{$(W,S)$-type} $\tau(\gamma)$ as the word 
    \[\tau(\gamma):=s_{j_1}\hdots s_{j_n}.\]
 
\end{definition}

\begin{definition}
    Given a combinatorial gallery $\gamma=(c_0,p_1,c_1,p_2,\hdots  ,p_n,c_n)$ of $\aw$, we say that $\gamma$ is \ix{folded} (or \ix{stammering}) if, within $\gamma$, we can find an index $i$ such that $c_i=c_{i-1}$. Then we say that $\gamma$ has a \ix{fold} at panel $p_i$. If there are no folds in $\gamma$, we say that $\gamma$ is \ix{unfolded} (or \ix{non-stammering}).  
\end{definition}

\begin{definition}
    Given a gallery $\gamma=(c_0,p_1,c_1,p_2,\hdots  ,p_n,c_n)$ of length $n$, define the \ix{fold set} $F(\gamma)$ to be the subset of $\{1,\hdots ,n\}$ such that $i\in$ $F(\gamma)$ if and only if $\gamma$ has a fold at panel $p_i$. 
\end{definition}

We can represent a gallery in the Coxeter complex by drawing a path starting at $c_0$ and passing through every alcove in the gallery. We represent the direction of the gallery with an arrow. 

\begin{definition}
The footprint $\mathrm{ft}(\gamma)$ of $\gamma$ is the gallery obtained by deleting all pairs $p_i,c_i$ such that $\gamma$ has a fold at $p_i$. 
\end{definition}

\begin{figure}[!htbp]
    \begin{center}
    \includegraphics[scale=0.2]{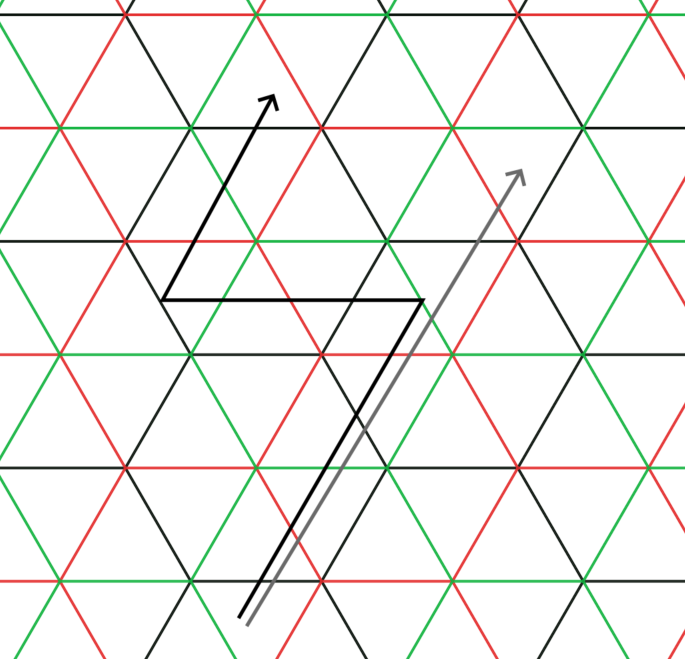}
    \end{center}
    \caption{Two galleries in $\aw$. The grey gallery is unfolded, and the black gallery is folded at two panels \cite{MASTERS}.}
\end{figure}

We now recall how to fold and unfold these galleries. Figure \ref{folds} provides an example.

\begin{figure}[!htbp]
    \begin{center}
    \includegraphics[scale=0.25]{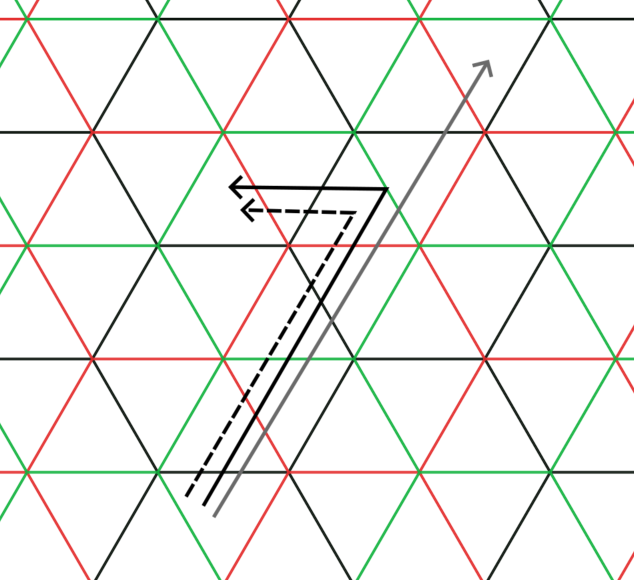}\\
    \end{center}
    \caption{Here we have a gallery shown in black. The grey line shows the corresponding unfolded gallery and the dashed line shows its footprint \cite{MASTERS}.}
    \label{folds}
\end{figure}

\begin{definition}
    Consider a gallery $\gamma = (c_0,p_1,c_1,\hdots ,p_n,c_n)$. Let $H_i$ be the wall containing the panel $p_i$, and let $r_i$ be the reflection across $H_i$. For $i=1,\hdots ,n$, let
    \[\gamma^i:=(c_0,p_1,\hdots ,p_i,r_ic_i,r_ip_{i+1},r_ic_{i+1},\hdots ,r_ip_n,r_ic_n).\]
    If $\gamma$ was folded at panel $p_i$, we call $\gamma^i$ an \ix{unfolding of }$\gamma$ at $p_i$. Otherwise, we call $\gamma^i$ a \ix{folding} of $\gamma$.
\end{definition}

\begin{lemma}\cite[Lemma 4.16]{SHA}
    For all $i=1,\hdots ,n$, $\tau(\gamma)=\tau(\gamma^i)$. So folding and unfolding does not change the gallery $(W,S)$-type. Also, $(\gamma^i)^i=\gamma$.
\end{lemma}

\begin{lemma}\cite[Lemma 4.17]{SHA}
    For all $i,j\in\{1,\hdots,n\}$, $(\gamma^i)^j=(\gamma^j)^i$, i.e.\ foldings commute.
\end{lemma}

\begin{corollary}
    Given any gallery $\gamma$, there is a subset $J\subset \{1,\hdots ,n\}$ such that $\gamma^J$ is unfolded, and $\gamma$ and $\gamma^J$ have the same $(W,S)$-type.
\end{corollary}

Now we fix a bijection between $W$ and the alcoves of our Coxeter complex of $\aw$. We denote the alcove associated to the identity element by 1.

\begin{definition}
    Let $w,x\in W$. Let $\gamma$ be any minimal gallery starting in 1 and ending in $w$. We write $w\rightharpoonup x$ if there is a folding $\gamma^I$ of $\gamma$ such that the end alcove of $\gamma^I$ is $x$.
\end{definition}

\begin{remark}
    This definition does not depend on the choice of minimal gallery. This is because the trivial orientation, where all folds are allowed, is braid invariant. See \cite[Chapter 5]{SHA} for details. 
\end{remark}

Given an element $w\in W$, we take any gallery $\gamma$ from 1 to $w$. We then define the shadow of an element as the set of all endpoints of foldings of $\gamma$. 
\begin{definition}
    Let $w\in \aw$. The \ix{(trivial) shadow} of $w$ is the set 
    \[\textnormal{Sh}(w):=\{x\in W\mid w\rightharpoonup x\}.\]
\end{definition}

\section{The coordinate system}\label{exploring}

Now we define new concepts for the $\aw$ Coxeter complex. We tackle this problem from a few different perspectives. In this section, we place coordinates on our Coxeter complex. In later sections, we show that, given a reduced word in the generators, there is an algorithm to output the corresponding coordinates of the respective alcove. We show that this new notation distinguishes the type of each alcove, i.e.\ the arrangement of panels in the alcove. Lastly, we define channels in the complex, and prove some foundational results about these channels. We use type to refer to both the standard type of a gallery, and the $(W,S)$-type. It will be clear from context which we mean, as the type is always a word in the indexing set, whilst the $(W,S)$-type is always a word in the generators of the group.

Recall that the presentation of $\aw$ is 
\[\aw=\langle s_0,s_1,s_2\mid s_i^2=1, (s_0s_1)^3=(s_0s_2)^3=(s_1s_2)^3=1\rangle,\]
with corresponding Coxeter complex the tiling of $\mathbb{R}^2$ by equilateral triangles. In our pictures of this complex, we represent the type of a panel by a colour. Here, we show $s_0$-type panels in black, $s_1$-type panels in red, and $s_2$-type panels in green. 

We let $\alpha$ and $\beta$ be a choice of simple roots corresponding to the reflections $s_2$ and $s_1$ respectively in our Coxeter complex. This fixes three directions, $\alpha, \beta$ and $\alpha+\beta$, as in Figure \ref{fig1}.

\begin{figure}[!htbp]
    \centering
    \def\svgwidth{0.4\textwidth}
    
    \resizebox{0.3\linewidth}{!}{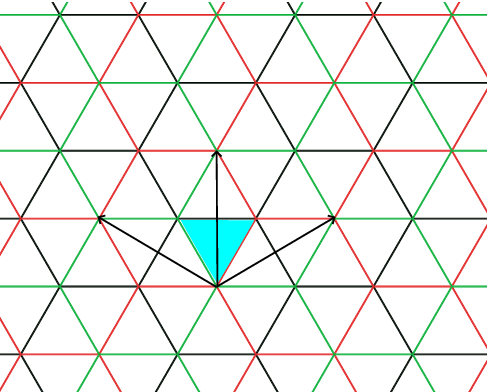}
    \caption{The directions of the roots.}
    \label{fig1}
\end{figure}

We use these roots to define tunnels and channels in our complex, by using them to distinguish the sets of parallel hyperplanes.

We make the following definition to help us explain the coordinate system and results.
\begin{definition}
An alcove $w$ is \ix{downwards facing} if $\ell(w)\equiv 0 \pmod 2$, and is otherwise \ix{upwards facing}. 

\end{definition}

\begin{remark}
    We fix the orientation of our pictures such that the identity alcove is downwards facing. Therefore, an upwards facing alcove is an alcove with a point towards the top of the page. Similarly, a downwards facing alcove is an alcove with a point towards the bottom of the page. See Figure \ref{upanddown}.

\begin{figure}[!htbp]
    \begin{center}
    \includegraphics[scale=0.07]{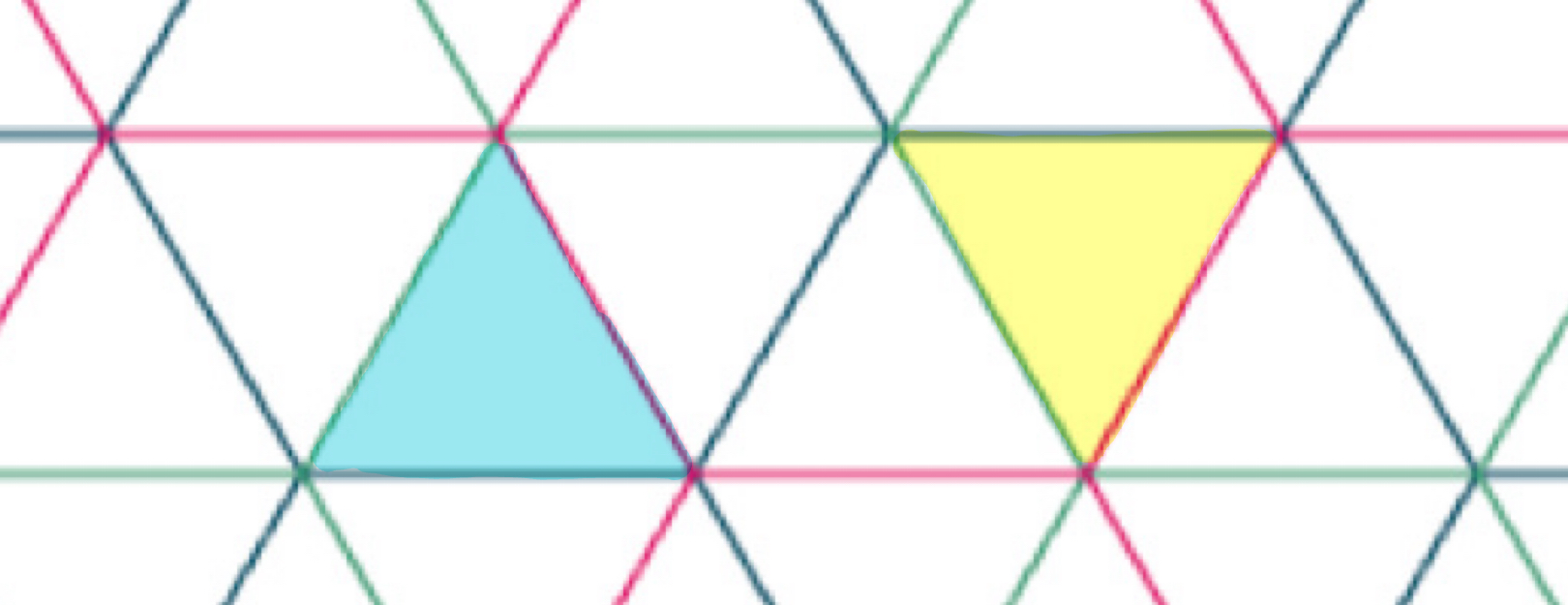}
    \end{center}
    \caption{An upwards facing alcove in blue, and a downwards facing alcove in yellow.}
    \label{upanddown}
\end{figure}

We also arbitrarily fix the $s_1$-type panel of the identity to be on the right, and the $s_2$-type panel to be on the left. 
\end{remark}

\subsection{Coordinate system}

We first define our coordinate system on $\aw$.
We describe our coordinate system by travelling through sets of alcoves called tunnels. The aim is to replace geometric positions of alcoves by pairs of integers that can be computed directly from reduced words. These objects sometimes arise in the literature as $J$-alcoves, for some $J\subset I$. See for instance \cite{GUILHOT}. 
\begin{definition}\label{tunnelsdef}
    Fix a root $\gamma$. A $\gamma$-\ix{tunnel} in the $\aw$ Coxeter complex is the set of alcoves contained between two adjacent parallel hyperplanes $H_{\gamma,n}$ and $H_{\gamma,n+1}$ that are perpendicular to $\gamma$.
\end{definition}

We place coordinates on the $\aw$ Coxeter complex. We need two variables to define the whole set of alcoves. We first give an intuitive geometric description of the coordinates. Later, we will see the algebraic description of these coordinates.

Firstly, we assign the identity alcove the coordinates $(0,0)$. Now fix $(n,m)$ with $n$ even. Now we describe which alcove this corresponds to. The first coordinate $n$ tells us how many alcoves to travel in the $\beta$-tunnel containing the identity, starting from the identity alcove
with positive moving up the page and negative moving down the page, as shown in Figure \ref{fig4}.
\begin{figure}[!htbp]
    \centering
    \def\svgwidth{0.4\textwidth}
    
    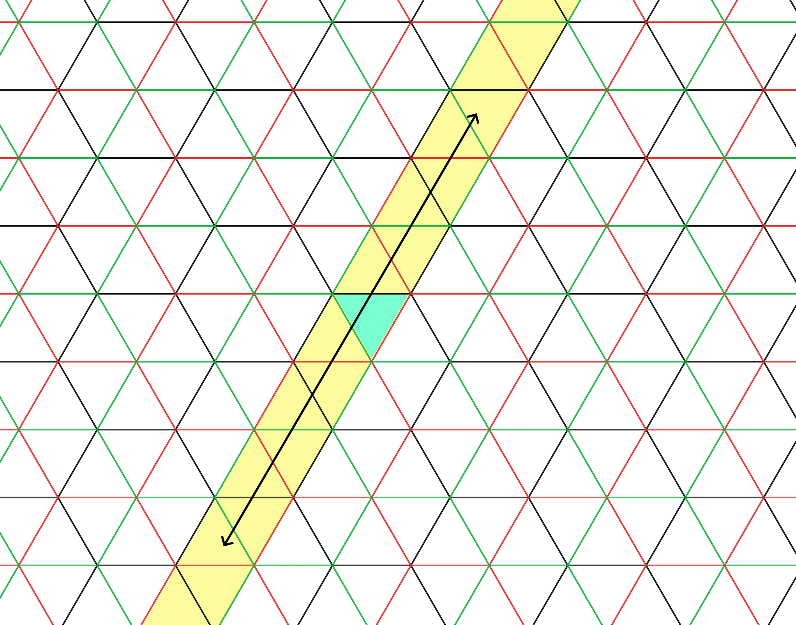
    \caption{The direction of the first coordinate.}
    \label{fig4}
\end{figure}

We always move an even number of alcoves, so we always end up in a downwards alcove, the same direction as the identity alcove.

The first coordinate $n$ tells us which $(\alpha+\beta)$-tunnel we are in. So the $(\alpha+\beta)$-tunnels are distinguished by the first coordinate.

Then we move in the $(\alpha+\beta)$-tunnels. If $m$ is positive, we move $m$ steps to the right in this tunnel. If $m$ is negative, we move $-m$ steps to the left in this tunnel.

\begin{figure}[!htbp]
    \centering
    \def\svgwidth{0.4\textwidth}
    
\resizebox{0.4\linewidth}{!}{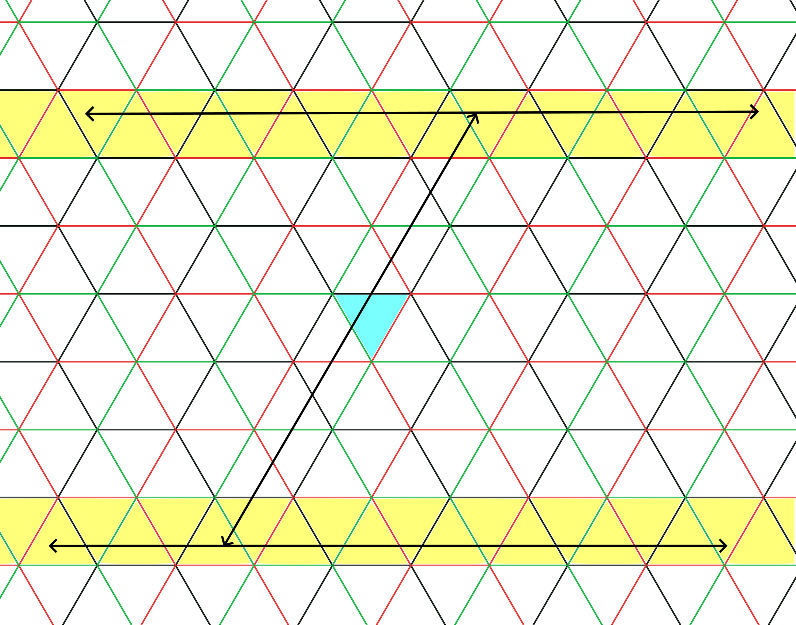}
    \caption{The direction of the second coordinate.}
    \label{fig5}
\end{figure}

For each set of coordinates $(n,m)$, we have created a gallery in our complex. We note that this is not always a minimal gallery. But its $(W,S)$-type is a word in our generators. 

To simplify the notation, let us introduce new notation to represent specific strings of generating elements. We will see later that these strings of elements play a key role in defining an algorithm for the coordinates. In the following definition, all subscripts are taken modulo 3. 

\begin{definition}
    We call a word in the generators \ix{increasing} if it is of the form \[s_is_{i+1}s_{i+2}s_is_{i+1}s_{i+2}\hdots,\] starting at some index $i$. Similarly, we call a word in the generators \ix{decreasing} if it is of the form $s_is_{i+2}s_{i+1}s_is_{i+2}s_{i+1}\hdots$ for some index $i$.
\end{definition}

\begin{remark}
    It is important to note here that we are increasing (or decreasing) the subscripts modulo 3, with no elements skipped. So $s_2s_0s_1s_2s_0$ is an increasing subword, but $s_0s_2s_0s_2$ is neither increasing nor decreasing.
\end{remark}

\begin{definition}
    Define $\gamma_i(n)$ to be the increasing word starting with the element $s_i$, of length $n$. Similarly, define $\delta_j(m)$ to be the decreasing word starting with the element $s_j$, of length $m$.
\end{definition}

\begin{remark}
    These increasing and decreasing sequences also appear in the definition of the normal form of an element. See \cite{HECKECATEGORY}.
\end{remark}
We prove the following two lemmas, which show that moving within any $(\alpha+\beta)$-tunnel has a specific type of gallery. 
\begin{lemma}\label{tunnels1}
    Let $w\in \aw$ and $i\neq j$. Then $w$ and $ws_is_js_i$ do not lie in the same tunnel. 
\end{lemma}

\begin{proof}
    Suppose they do. Then 1 and $s_is_js_i$ lie in the same tunnel, as every tunnel is mapped to a tunnel under the action of any element $x\in W$ on the complex. Say they are bounded by parallel adjacent hyperplanes $H_1$ and $H_2$. But now $s_is_js_i$ is the longest element in the parabolic subgroup $W_{\{i,j\}}$, which is the group generated by the elements $s_i$ and $s_j$. So every hyperplane passing through the shared point of 1 and $s_is_js_i$ must separate 1 and $s_is_js_i$. But, in $\aw$, each vertex has every direction of walls passing through it. So there is a hyperplane parallel to $H_1$ and $H_2$ separating 1 and $s_is_js_i$, contradicting the adjacency of $H_1$ and $H_2$. 
 \end{proof}

\begin{lemma}\label{tunnels}
    Any minimal gallery moving only within an $(\alpha+\beta)$-tunnel to the right has type 
    $\delta_i(n)$.
    Any minimal gallery moving only within an $(\alpha+\beta)$-tunnel to the left has type 
    $\gamma_i(n)$.
\end{lemma}

\begin{proof}
        We prove that this is true for the $(\alpha+\beta)$-tunnel containing the origin. Then every other $(\alpha+\beta)$-tunnel is a combination of reflections of the origin $(\alpha+\beta)$-tunnel. So, as the panel types are preserved under reflections in the hyperplanes, the statement must be true for all $(\alpha+\beta)$-tunnels.

    For the $(\alpha+\beta)$-tunnel containing the origin, we proceed by induction, proving it to be true for any minimal gallery starting at the origin. Then it must be true for any minimal gallery in the tunnel, as the gallery must be a subset of a minimal gallery from the origin, a reversal of a minimal gallery from the origin, or a combination of the two. We first note that the statement is true when we step into the first two alcoves in either direction. 

    Now assume it is true for a minimal gallery of length $k>1$ moving in the right direction. Assume we have a minimal gallery starting in the identity alcove of length $k+1$. By induction, our gallery has type $\delta_1(k)s_i=\delta_1(k-2)s_{-k}s_{2-k}s_i$. Note that, as the gallery is minimal, $i\neq 2-k$. Now, by Lemma \ref{tunnels1}, if $i=-k$, then $\delta_1(k-2)$ and $\delta_1(k)s_i$ are not in the same tunnel. So $i\neq -k$, and hence $i=1-k$. Therefore, $\delta_1(k)s_i=\delta_1(k+1)$.

     A very similar argument holds for a minimal gallery in the left direction.
\end{proof}
This can then be applied to the other two types of panels.

\begin{lemma}\label{tunnelsbeta}
    Any gallery moving only within a $\beta$-tunnel to the right and up has type 
    $\gamma_i(n)$.
    Any gallery moving only within a $\beta$-tunnel to the left and down has type 
    $\delta_i(n)$.
\end{lemma}

\begin{proof}
    Reflecting a $\beta$-tunnel along any $\alpha$-hyperplane sends the $\beta$-tunnel to an $(\alpha+\beta)$-tunnel. The result follows from Lemma \ref{tunnels}.
\end{proof}

\begin{lemma}\label{tunnelsalpha}
    Any gallery moving only within an $\alpha$-tunnel to the right and down has type 
    $\gamma_i(n)$.
    Any gallery moving only within an $\alpha$-tunnel to the left and up has type 
    $\delta_i(n)$.
\end{lemma}
\begin{proof}
    Reflecting an $\alpha$-tunnel along any $\beta$-hyperplane sends the $\alpha$-tunnel to an $(\alpha+\beta)$-tunnel. The result follows from Lemma \ref{tunnels}.
\end{proof}

We can now use these results to turn the geometric description of the coordinate system into a definition on the elements of our Coxeter group. Let us see what word corresponds to the gallery $(n,m)$. Firstly, if $n$ and $m$ are nonnegative, we first cross over $n$ panels with repeating subword $s_0s_1s_2s_0s_1s_2\hdots$, ending with crossing over some panel, say $s_i$. This is because, by definition, we are travelling upwards in a $\beta$-tunnel. Then we start a new repeating subword $s_{i+2}s_{i+1}s_{i}s_{i+2}s_{i+1}s_{i}\hdots$ of length $m$, where the indices are taken modulo 3. Note that now the indices are decreasing modulo 3, because we are moving to the right in an $(\alpha+\beta)$-tunnel. So we have the word
\[s_0s_1s_2s_0s_1s_2\cdots s_is_{i+2}s_{i+1}s_{i}s_{i+2}\cdots .\]
We write this in our notation as 
\[\gamma_0(n)\delta_{n+1}(m).\]

Now, if $n$ is negative and $m$ is nonnegative, we first cross over $n$ panels with repeating subword $s_2s_1s_0s_2s_1s_0\hdots$, ending with crossing over some panel, say $s_i$. Then we skip the panel $s_{i+1}$ and continue the same pattern $s_{i+2}s_is_{i+1}s_{i+2}s_is_{i+1}\hdots$ for $m$ panels. So now we have the word
\[s_2s_1s_0s_2s_1s_0\cdots s_is_{i+2}s_is_{i+1}s_{i+2}\cdots, \]

which in our notation is
\[\delta_2(n)\delta_{2-n}(m).\]

Similarly, if $n$ is nonnegative and $m$ is negative, we get the word
$\gamma_0(n)\gamma_{n+2}(m)$. Lastly, if $n$ and $m$ are negative, we get the word
$\delta_2(n)\gamma_{-n}(m)$. So, overall, our coordinates correspond to the galleries, and therefore the elements of $W$, as follows:
\[(n,m)=\begin{cases}
    \gamma_0(n)\delta_{n+1}(m) & \text{ if } n\geq 0,m\geq 0,\\
    \delta_2(-n)\delta_{1-n}(m) & \text{ if } n<0,m\geq 0,\\
    \gamma_0(n)\gamma_{n+2}(-m)& \text{ if } n\geq 0,m<0, \\
    \delta_2(-n)\gamma_{2-n}(-m)& \text{ if } n<0,m< 0.\\
\end{cases}\]

Throughout the paper, we refer to the notation $(n,m)$ as both an alcove in the Coxeter complex, and so an element of the Coxeter group, and a word in our generators of the Coxeter group, and so a gallery in the Coxeter complex. 

\subsection{Corresponding sections of the complex}

 We now split our $\aw$ complex into 3 sections, with our identity element not in any section. We use these sections to define the algorithm for our coordinate system.

\begin{center}
\resizebox{0.5\textwidth}{!}{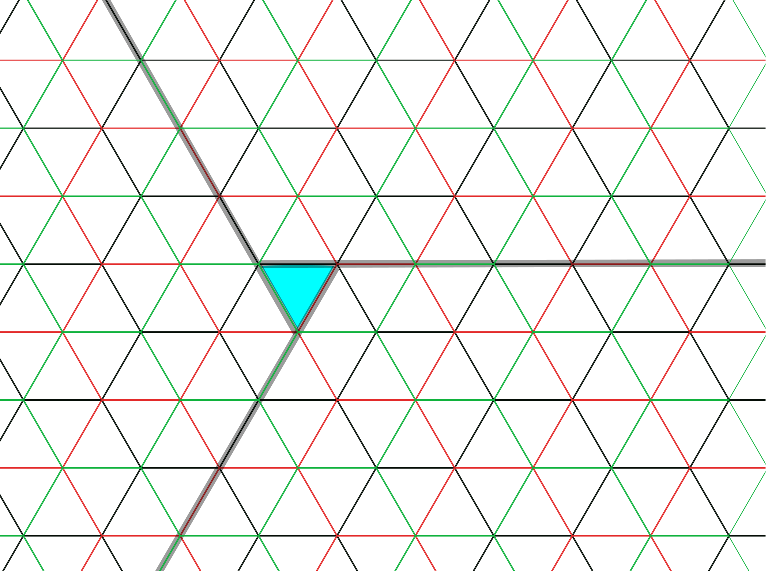}
\end{center}

We have
\begin{align*}
    s_0 \text{-region}&=\{(n,m)\mid n>0, m\geq -n\},\\
    s_1 \text{-region}&=\{(n,m)\mid n\leq 0, m> 0\},\\
    s_2 \text{-region}&=\{(n,m)\mid n\leq 0, m\leq 0, (n,m)\neq (0,0)\}\cup\{(n,m)\mid n> 0, m<-n\}.
\end{align*}

\subsection{Odd $n$ notation}\label{involution}

Now we define a slightly different notation, which has an odd first coordinate. Here, the first coordinate $n$ is similar to the previously described coordinate system, as it is still describing how many alcoves we move in the $\beta$-tunnel.

Now, though, we are moving an odd number of alcoves, and so we end up in an upwards facing alcove.  Then we move in the $\alpha$-tunnels.
\begin{center}
\includegraphics[scale=0.5]{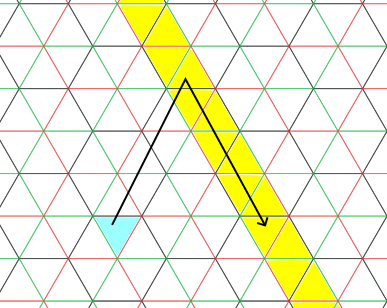}
\end{center}

We move up and to the left if our second coordinate is positive, and down and to the right if our second coordinate is negative. 

One key point to notice is that, if we represented an alcove with coordinates $(n,0)$, this alcove can also be represented by the coordinates $(n+1,-1)$.

So now each alcove can be represented by two sets of coordinates - when $n$ is even and when $n$ is odd. We have a map between these two sets of coordinates, given by the next lemma. 

Throughout this paper, for any integer $k$, we write
\[
\epsilon_k=
\begin{cases}
1 & \text{if } k \text{ is odd},\\
0 & \text{if } k \text{ is even}.
\end{cases}
\]

\begin{lemma}
    Any alcove in $\aw$ has exactly two sets of coordinates. They are related by the involution 
    \[(n,m)\mapsto (n+m+1-\epsilon_m,-m-1). \]
\end{lemma}

\begin{proof}
    We first note that it is an involution, by calculating
    \begin{align*}
        (n+m+1-\epsilon_m,-m-1)\mapsto& ((n+m+1-\epsilon_m)+(-m-1)+1-\epsilon_{-m-1},-(-m-1)-1)\\=&(n+1-(\epsilon_m+\epsilon_{-m-1}),m),
    \end{align*}
    and noting that $\epsilon_m+\epsilon_{-m-1}=1$ for every $m$. 

    Next, we show that $(n,m)$ and $(n+m+1-\epsilon_m,-m-1)$ represent the same alcove in $\aw$. Here we consider cases. We draw the picture of the galleries of both coordinate systems, and use the basic geometric facts from Lemma \ref{triangles} to prove that the map between the two sets of coordinates is correct. Now we have shown that our map is an involution, we only need to show that the map is correct when moving from the even notation to the odd notation.

    Firstly, let us consider an alcove with our even coordinates $(n,m)$, and first assume that $n$ is positive and that $m$ is even and positive. Let $(x,y)$ be the odd coordinates of this alcove. Let each edge of the triangles have length 1. Then we have the following picture.

\begin{center}
\resizebox{0.4\linewidth}{!}{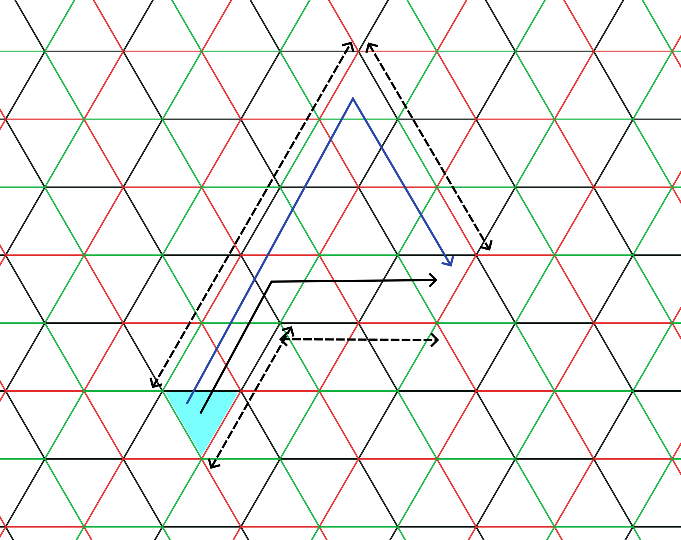}
\end{center}

The black gallery is the gallery described by the even coordinates, and the blue gallery is the gallery described by the odd coordinates. Now we can see that 
\[\dfrac{x+1}{2}=\dfrac{n}{2}+\dfrac{m}{2}+1,\]
and so 
\[x=n+m+1.\]
Similarly, 
\[\dfrac{-y+1}{2}=\dfrac{m}{2}+1,\]
and so 
\[y=-m-1.\]
 Now if $m$ is odd, then $(n,m-1)$ is the case above. So the odd coordinates of $(n,m-1)$ are $(n+m,-m-2)$. Then, in the odd coordinate notation, to move to $(n,m)$ we keep the first coordinate the same and add one to the second coordinate. So the odd coordinates of $(n,m)$ are $(n+m,-m-1)$, which again agrees with the involution.

The other regions are proven similarly and can be verified by the reader. 

We note that it is clear from the definitions of each notation that there is a unique set of coordinates for each alcove when we fix the parity of $n$. So the two sets of coordinates are the only two representing the alcove. 
Therefore, there are exactly two representations for each alcove in $\aw$. 
\end{proof}

\begin{remark}
    We note that the involution always changes the sign of the second coordinate, between negative and nonnegative. Also, the involution changes the parity of the second coordinate.
\end{remark}

Within our results, we use both of these sets of coordinates. This allows us to avoid a lot of case handling. In any result, if there is a set of coordinates $(n,m)$, this is defined using the first notation if $n$ is even, and the second notation if $n$ is odd.

\section{The Words-to-Coordinates algorithm}\label{algo}
We now develop the \hyperref[algorithm]{Words-to-Coordinates Algorithm} to convert reduced words to their coordinates. This algorithm gives the even coordinates of the word. Using the involution above, we can obtain the odd coordinates. First note that again all $s_i$ subscripts are defined modulo 3. 

We take a word $w$, and assume that it is reduced. Starting from the left, we break the word down into increasing and decreasing subwords that are as long as possible. Note again that the subscripts must be adjacent modulo 3, we cannot skip elements in the subword. 

\begin{algorithm}[Words-to-Coordinates Algorithm]\label{algorithm}
    Define an algorithm mapping reduced words of elements in $\aw$ to coordinates as follows. 
    
    The identity word is assigned coordinates $(0,0)$. Now assume that we have a nonempty reduced word $w$ starting with $s_i$.
    \begin{enumerate}
        \item  First write $w$ as $w=w_1w_2$, such that $w_1$ is a decreasing subword of maximal length.
    \end{enumerate}
    The next step depends on the lengths of $w_1$ and $w_2$.
    \begin{itemize}
        \item[ - ] If $\ell(w_1)$ is even and $\ell(w_2)>0$,  set $j\equiv i-1\pmod 3$. 
        \begin{enumerate}
            \setcounter{enumi}{1}
            \item Begin decomposing $w$. Start with a decreasing subword of maximal length. 
            \item Continue dividing $w$ into alternating increasing and decreasing subwords of maximal length. So write the word $w$ as 
            \[w=\delta_{i_1}(n_1)\gamma_{i_2}(n_2)\delta_{i_3}(n_3)\gamma_{i_4}(n_4)\hdots,\]
             with $n_k>0$ for each $k$.
             \item Set $u=\sum_{k \text{ even}}n_k$, so $u$ is the number of elements in increasing subwords. 
        \end{enumerate}

    \item [ - ]
    Otherwise, set $j=i$. 
    \begin{enumerate}
    \setcounter{enumi}{1}
        \item Begin Decomposing $w$. Start with an increasing subword of maximal length, even if that means it is of length one.
        \item Continue dividing the word into alternating increasing and decreasing subwords of maximal length.  So write the word $w$ as \[w=\gamma_{i_1}(n_1)\delta_{i_2}(n_2)\gamma_{i_3}(n_3)\delta_{i_4}(n_4)\hdots,\]
    with $n_k>0$ for each $k$. 
    \item Set $u=\sum_{k \text{ odd}}n_k$, so again $u$ is the number of elements in increasing subwords.  
    \end{enumerate}
    \end{itemize}

\begin{enumerate}
\setcounter{enumi}{4}
    \item After $w$ has been decomposed, if the word ends in an increasing subword, and the length of that end subword is odd, set $x=u+1$ and $y=\ell(w)-u-1$, where $\ell(w)$ is the length of the word $w$. Otherwise, let $x=u$ and $y=\ell(w)-u$.
\item The coordinates of $w$ are then
\[(n,m)=\begin{cases}
    (x,y)& j=0,\text{ and }x \text{ even}\\
    (x+y+1-\epsilon_y,-y-1)& j=0, \text{ and }x \text{ odd},\\
    (-x-y+\epsilon_y,x-\epsilon_y)& j=1, \text{ and }x \text{ even},\\
    (1-x,x+y)& j=1, \text{ and }x\text{ odd},\\
    (y+\epsilon_{y},-x-y)& j=2, \text{ and } x \text{ even},\\
    (-y-\epsilon_y,-x+\epsilon_y)& j=2, \text{ and } x \text{ odd}.\\
\end{cases}\]
\end{enumerate}

\end{algorithm}

\begin{remark}
    Here, the variable $j$ is telling us that the element $w$ lies in the $s_j$-region. We will see later that this algorithm correctly identifies the regions. 
\end{remark}

\begin{remark}\label{timecomplexity}
The \hyperref[algorithm]{Words-to-Coordinates Algorithm} runs in linear time with respect to the length of the input word. Indeed, each generator is examined at most twice while decomposing the word into maximal increasing and decreasing subwords. During this process, we only keep track of the lengths of the current subword, the total number of generators contained in increasing subwords, and the parity of the final subword. Once this decomposition has been completed, the remaining calculations involve only a fixed number of arithmetic operations. Hence, if the input word has length $\ell(w)$, the overall time complexity of the algorithm is $O(\ell(w))$.
\end{remark}

\begin{example}
    Given the word
    \[w=s_0s_1s_2s_0s_2s_1s_2s_0s_1s_2s_1,\]
    the longest decreasing subword at the beginning of the word is $s_0$. This is of odd length, so this subword lies in the $s_0$-region. We now see that the first maximal increasing subword is $s_0s_1s_2s_0$. Then we have a decreasing subword $s_2s_1$. Then we have another increasing subword $s_2s_0s_1s_2$. Lastly, we have a decreasing subword of $s_1$. So we have written this word as 
    \[w=\gamma_0(4)\delta_2(2)\gamma_2(4)\delta_1(1).\]
    So we have 8 elements in increasing subwords, and therefore $u=8$. Now this word does not end in an odd increasing subword, and $j=0$, so $(n,m)=(8,3)$. 
    \begin{center}
\end{center}
\end{example}

\begin{example}
    Given the word 
    \[v=s_2s_1s_0s_2s_0s_1s_0s_2s_0,\]
    we start with a decreasing subword of $s_2s_1s_0s_2$. So we know that this word lies in the $s_1$-region and $j=1$. So we begin the algorithm with the decreasing subword of $s_2s_1s_0s_2$. Then we have an increasing subword of $s_0s_1$. We then have a decreasing subword of $s_0s_2$. Lastly, we finish with the increasing subword $s_0$. Therefore, we have written our word as 
    \[v=\delta_2(4)\gamma_0(2)\delta_0(2)\gamma_0(1).\]
    So we have $u=3$. Now this word does end in an odd increasing subword, so we have $x=4$ and $y=5$. This gives us $(n,m)=(-8,3)$.

    \begin{center}

\end{center}
\end{example}

\subsection{Proof of algorithm}
Given a reduced word $w$, let us denote by $(n_w,m_w)$ the output of the \hyperref[algorithm]{Words-to-Coordinates Algorithm}. We now prove that the algorithm does indeed give us the correct element in the Coxeter group. Equivalently, we also prove that the gallery of any reduced word $w$ has the same end alcove as the gallery described by the corresponding output of the algorithm $(n_w,m_w)$.

Our method to prove this theorem involves showing that the algorithm is constant under the braid relations and that it holds for some given set of reduced words that cover the whole group. As we have restricted our algorithm to reduced words, this proves that the algorithm outputs the correct coordinates for any reduced word.

First, we show that the region calculated in the algorithm is constant under the braid relations. Later, we will show that the algorithm outputs the correct region for some standard reduced words. Combining these results, we will be able to conclude  that the algorithm outputs the correct region for any reduced word.

\begin{proposition}\label{regions}
    If $w,v\in W$ are two reduced words that differ by a braid relation, then the region produced by the algorithm for $w$ and $v$ is the same region. 
\end{proposition}
\begin{proof}
    We can assume without loss of generality that $w$ contains the braid relation $s_js_{j+1}s_j$ for some $j\in\{0,1,2\}$.
    Let us write $w=\delta_i(d)w_2$ and $v=\delta_k(e)v_2$, with $d$ and $e$ maximal. We note that the word $\delta_i(n)$ contains no braid relation. So  $w\neq \delta_i(n)$, and therefore we can assume that $\ell(w_2)>0$. By the definition of a decreasing subword, the braid relation in $w$ cannot completely appear in the decreasing subword $\delta_i(d)$. Furthermore, we cannot have the first two elements of the braid relation appear in the decreasing subword of $w$, as $s_js_{j+1}$ is not decreasing. Lastly, if the first element appears in the decreasing subword, and $d>1$, then $w$ must look like 
    \[w=\delta_i(d-2)s_{j+1}s_js_{j+1}s_j\hdots,\]
    and so, applying the braid relation, we can clearly see that this is not reduced. 
    
     Therefore, the braid relation must be contained in $w_2$, or $d=1$ and $w$ begins with the braid relation. Assume that $w_2$ contains the braid relation. Let us first consider the case that $w_2$ does not begin with the braid relation. Then, this would not affect the algorithm and so the region produced would be unchanged. So we are left with two cases. Firstly, if $w$ begins with the braid relation, and secondly if the braid relation occurs directly after the decreasing subword.

    Case 1: $w=s_js_{j+1}s_j\hdots$, and so 
\[w=\underbrace{s_j}_\downarrow\underbrace{s_{j+1}}_\uparrow\underbrace{s_j\hdots}_\downarrow\hdots,\]
    so $d$ is odd and the algorithm tells us that $w$ is in the $s_j$-region.
    Now 
    \[v=\underbrace{s_{j+1}s_j}_\downarrow\underbrace{s_{j+1}\hdots}_\uparrow\hdots,\]
    so $e$ is even, and the algorithm tells us that $v$ still lies in the $s_{(j+1)-1}=s_j$-region. 

    Case 2: $w=\delta_i(d)s_js_{j+1}s_j\hdots$ for $d>0$, and so, as $d$ is maximal and $w$ is reduced, the decreasing subword must end in $s_{j+2}$. So
    \[w=\underbrace{\delta_i(d-1)s_{j+2}}_\downarrow s_js_{j+1}s_j\hdots,\]
    and applying the braid relation we get
    \[v=\underbrace{\delta_i(d-1)s_{j+2}s_{j+1}s_j}_\downarrow s_{j+1}\hdots,\]
    and hence $e=d+2$, and so the parity is unchanged. As the starting element of $w$ and $v$ is the same, the algorithm outputs the same region.
\end{proof}

We now show that the algorithm is constant under braid relations. 

\begin{proposition} \label{regionslemma}
    Given two reduced words $w,v\in W$ that differ by a braid relation, we have $n_w=n_v$ and $m_w=m_v$. 
\end{proposition}

\begin{proof}
    The three braid relations we have are 
    \[s_0s_1s_0\longleftrightarrow s_1s_0s_1,\]
    \[s_1s_2s_1\longleftrightarrow s_2s_1s_2,\]
    \[s_2s_0s_2\longleftrightarrow s_0s_2s_0.\]
    We can assume without loss of generality that $w=w_1s_is_{i+1}s_iw_2$ and $v=w_1s_{i+1}s_is_{i+1}w_2$ for some reduced, possibly empty, words $w_1, w_2$ and $i\in \{0,1,2\}$. We proceed by considering cases where $w_1$ and $w_2$ are possibly empty. We apply the algorithm to $w$, outputting coordinates $(n_w,m_w)$. In some of the cases, we further split into two subcases:  the case that the start of the braid relation is contained in an increasing subword and the case that the start is contained in a decreasing subword when the algorithm is applied.
    
    Case 1: $w_1$ and $w_2$ are empty. Then $w=s_is_{i+1}s_i$. So the algorithm tells us that $w$ lies in the $s_i$ region. Now, running the algorithm, we get 
        \[w=\underbrace{s_is_{i+1}}_\uparrow \underbrace{s_i}_\downarrow,\]
        and so $u_w=2$, and $x_w=2$, $y_w=1$. Then applying the braid relation we get 
        \[v=\underbrace{s_{i+1}s_i}_\downarrow \underbrace{s_{i+1}}_\uparrow,\]
        as we know that we are in the $s_i$-region  and we are now starting with $s_{i+1}$, so we start with a decreasing subword. Now we have $u_v=1$, but we now have a word ending in an odd increasing subword. So $x_v=2$, $y_v=1$, and hence $n$ and $m$ are preserved, as, by Proposition \ref{regions}, the region determined by the algorithm is constant under braid relations.

    Case 2: $w_1$ is empty but $w_2$ is nonempty. 
    So there is at least one more element after $s_is_{i+1}s_i$. Then this element must be $s_{i+2}$; if the element was $s_i$ we would immediately see that it is not a reduced word, and if it was $s_{i+1}$ we could apply the braid relation and again see it was not a reduced word. Now we need to consider three subcases, depending on the length of $w_2$ and the element after $s_{i+2}$ in $w_2$.

    \begin{enumerate}[(a)]
        \item $\ell(w_2)=1$, and so $w_2=s_{i+2}$. Then we have
        \[w=\underbrace{s_is_{i+1}}_\uparrow \underbrace{s_is_{i+2}}_\downarrow,\]
        and applying the braid relation we have
        \[v=\underbrace{s_{i+1}s_i}_\downarrow \underbrace{s_{i+1}s_{i+2}}_\uparrow,\]
        as we are now starting with $s_{i+1}$, so we start with a decreasing subword. These both give $x_w=x_v=2$, $y_w=y_v=2$. 
        
        \item $\ell(w_2)>1$, and $w_2=s_{i+2}s_{i+1}\hdots$, with the dots possibly representing an empty string. Then we have
        
        \[w=\underbrace{s_is_{i+1}}_\uparrow \underbrace{s_is_{i+2}s_{i+1}\hdots}_\downarrow\hdots,\]
        and applying the braid relation we have
        \[v=\underbrace{s_{i+1}s_i}_\downarrow \underbrace{s_{i+1}s_{i+2}}_\uparrow\underbrace{s_{i+1}\hdots}_\downarrow\hdots,\]
        as again we are now starting with $s_{i+1}$. Now the number of elements in an increasing subword is preserved, and we have not changed the end subword, and so $n$ and $m$ are preserved.
        \item $\ell(w_2)>1$, and $w_2=s_{i+2}s_i\hdots$, with the dots possibly representing an empty string. Then we have
        
        \[w=\underbrace{s_is_{i+1}}_\uparrow \underbrace{s_is_{i+2}}_\downarrow\underbrace{s_i\hdots}_\uparrow\hdots,\]
        and applying the braid relation we have
        \[v=\underbrace{s_{i+1}s_i}_\downarrow \underbrace{s_{i+1}s_{i+2}s_i\hdots}_\uparrow\hdots,\]
        as again we are now starting with $s_{i+1}$. Now the number of elements in an increasing subword is preserved, and, as we have not changed the parity of the length of a possible end increasing subword, $n$ and $m$ are preserved.
    \end{enumerate}
    Case 3: $w_1$ is nonempty but $w_2$ is empty. As $w_1$ is nonempty,  by the same argument as above, $w_1$ must end in $s_{i+2}$. So our $w$ looks like $w=\hdots s_{i+2}s_is_{i+1}s_i$. We now run the algorithm on this word. We consider two cases: when the $s_{i+2}$ element at the end of the subword $w_1$ is contained in an increasing subword, and when it is contained in a decreasing subword. 
\begin{enumerate}[(i)]
        \item $s_{i+2}$ is contained in an increasing subword. So we have
    \[w=\hdots\underbrace{\hdots s_{i+2}s_is_{i+1}}_\uparrow \underbrace{s_i}_\downarrow.\]
    Let $u$ be the number of elements in increasing subwords. Then we are not ending in an odd increasing subword, so $x_w=u$ and $y_w=\ell(w)-u$.
    
    Applying the braid relation we have 
    \[v=\hdots\underbrace{\hdots s_{i+2}}_\uparrow \underbrace{s_{i+1}s_i}_\downarrow \underbrace{s_{i+1}}_\uparrow.\]
    Now we note that everything up to, and including, our $s_{i+2}$ has not been changed, and so whether it is in an increasing or decreasing subword has not changed. Then, overall, the number of elements in an increasing subword has decreased by one, so we now have $u-1$ elements in increasing subwords. But we are now ending in an odd length increasing subword, so we still have $x_v=u$ and $y_v=\ell(w)-u$. Then our $n$ and $m$ are unchanged. 

    \item $s_{i+2}$ is contained in a decreasing subword. So we have 
    \[w=\hdots\underbrace{\hdots s_{i+2}}_\downarrow \underbrace{s_is_{i+1}}_\uparrow \underbrace{s_i}_\downarrow.\] 
    Applying the braid relation, we get
    \[v=\hdots\underbrace{\hdots s_{i+2}s_{i+1}s_i}_\downarrow \underbrace{s_{i+1}}_\uparrow.\]
    Again, the number of elements in an increasing subword has decreased by one. But our word now ends in an odd increasing subword, and so our $n$ and $m$ are the same before and after the braid relation.
    \end{enumerate}

    Case 4: Both $w_1$ and $w_2$ are nonempty. Then $w_1$ must end with $s_{i+2}$ and $w_2$ must start with $s_{i+2}$. So $w=\hdots s_{i+2}s_is_{i+1}s_is_{i+2}\hdots$, where the dots can be empty. Again, we consider the two cases when the $s_{i+2}$ element at the end of the subword $w_1$ is contained in an increasing subword, and when it is contained in a decreasing subword. Within these cases we need to consider three subcases, just as we did in Case 1, depending on the length of $w_2$.

    \begin{enumerate}[(i)]
        \item Our starting $s_{i+2}$ appears in an increasing subword.
        \begin{enumerate}
        \item $\ell(w_2)=1$, and so $w_2=s_{i+2}$. So we have
        \[w=\hdots\underbrace{\hdots s_{i+2}s_is_{i+1}}_\uparrow\underbrace{s_is_{i+2}}_\downarrow,\]
        which we can apply the braid relation to. Then we get
        \[v=\hdots\underbrace{\hdots s_{i+2}}_\uparrow\underbrace{s_{i+1}s_i}_\downarrow\underbrace{s_{i+1}s_{i+2}}_\uparrow,\]
        and we can clearly see that the number of elements in increasing subwords is preserved. The braid move now means that we end in an increasing subword. But this subword has even length. So overall the values of $n$ and $m$ are preserved under this braid move.

        \item $\ell(w_2)>1$, and $w_2=s_{i+2}s_{i+1}\hdots$, with the dots possibly representing an empty string. Then we have 
        \[w=\hdots\underbrace{\hdots s_{i+2}s_is_{i+1}}_\uparrow\underbrace{s_is_{i+2}s_{i+1}\hdots}_\downarrow\hdots,\]
        and, applying the braid relation, we get
        \[v=\hdots\underbrace{\hdots s_{i+2}}_\uparrow\underbrace{s_{i+1}s_i}_\downarrow\underbrace{s_{i+1}s_{i+2}}_\uparrow\underbrace{s_{i+1}\hdots}_\downarrow\hdots,\]
        which preserves $u$ and whether the word ends in an odd increasing subword. 

        \item $\ell(w_2)>1$, and $w_2=s_{i+2}s_i\hdots$, with the dots possibly representing an empty string. Then we have 
        \[w=\hdots\underbrace{\hdots s_{i+2}s_is_{i+1}}_\uparrow\underbrace{s_is_{i+2}s_{i+1}}_\downarrow\underbrace{s_i\hdots}_\uparrow\hdots,\]
        and, applying the braid relation, we get
        \[v=\hdots\underbrace{\hdots s_{i+2}}_\uparrow\underbrace{s_{i+1}s_i}_\downarrow\underbrace{s_{i+1}s_{i+2}s_i\hdots}_\uparrow\hdots,\]
        which preserves $u$ and the parity of any end subword.

        \end{enumerate}
        \item Our starting $s_{i+2}$ appears in a decreasing subword. Again consider the three subcases.
        \begin{enumerate}
            \item $\ell(w_2)=1$, and so $w_2=s_{i+2}$.
        \[w=\hdots\underbrace{\hdots s_{i+2}}_\downarrow\underbrace{s_is_{i+1}}_\uparrow\underbrace{s_is_{i+2}}_\downarrow,\]
        and applying the braid relation we get
        \[v=\hdots\underbrace{\hdots s_{i+2}s_{i+1}s_i}_\downarrow\underbrace{s_{i+1}s_{i+2}}_\uparrow.\]
        Now clearly the number of elements in an increasing subword is unchanged under the braid move. Also, we are now ending on an increasing subword, but its length is even. So again, $n$ and $m$ are preserved under the braid move. 

         \item $\ell(w_2)>1$, and $w_2=s_{i+2}s_{i+1}\hdots$, with the dots possibly representing an empty string. Then we have 
        \[w=\hdots\underbrace{\hdots s_{i+2}}_\downarrow\underbrace{s_is_{i+1}}_\uparrow\underbrace{s_is_{i+2}s_{i+1}\hdots}_\downarrow\hdots,\]
        and applying the braid relation we get
        \[v=\hdots\underbrace{\hdots s_{i+2}s_{i+1}s_i}_\downarrow\underbrace{s_{i+1}s_{i+2}}_\uparrow\underbrace{s_{i+1}\hdots}_\downarrow\hdots,\]
        so clearly again the number of elements in an increasing subword is unchanged under the braid move. So again, $n$ and $m$ are preserved under the braid move.

        \item $\ell(w_2)>1$, and $w_2=s_{i+2}s_i\hdots$, with the dots possibly representing an empty string. Then we have 
        \[w=\hdots\underbrace{\hdots s_{i+2}}_\downarrow\underbrace{s_is_{i+1}}_\uparrow\underbrace{s_is_{i+2}}_\downarrow\underbrace{s_i\hdots}_\uparrow\hdots,\]
        and applying the braid relation we get
        \[v=\hdots\underbrace{\hdots s_{i+2}s_{i+1}s_i}_\downarrow\underbrace{s_{i+1}s_{i+2}s_i\hdots}_\uparrow,\]
        and the number of elements in an increasing subword is unchanged under the braid move. Also, the parity of any end subword is constant. So again, $n$ and $m$ are preserved under the braid move.

        \end{enumerate}
    \end{enumerate}
    This covers all cases, and so the algorithm is constant under braid relations. 
\end{proof}

We now take a set of reduced words that cover the whole group, and prove the algorithm holds on these words. We start by proving that words of the form $\gamma_i(\nu)\delta_{\nu+i+1}(\mu)$ are reduced. We note that each generator has a corresponding root - $s_1$ corresponds to $\beta$, $s_2$ corresponds to $\alpha$, and $s_0$ corresponds to $-\alpha-\beta$. Let $\alpha_i$ be the root corresponding to $s_i$. 

\begin{lemma}\label{beforeturn}
Let
\[
w=s_0s_1s_2s_0s_1s_2\cdots.
\]
Let
\[
\beta_k=s_{i_1}\cdots s_{i_{k-1}}(\alpha_{i_k})
\]
be the $k$-th crossed root of $w$. Then every crossed root of $w$ is positive.
\end{lemma}

\begin{proof}
Let $c=s_0s_1s_2$. We first compute the first three crossed roots, using the formulae
\[
s_i(\alpha_i)=-\alpha_i,
\qquad
s_i(\alpha_j)=\alpha_j+\alpha_i
\quad (i\neq j).
\]
We obtain
\[
\beta_1=\alpha_0,
\]
\[
\beta_2=s_0(\alpha_1)=\alpha_0+\alpha_1,
\]
and
\[
\beta_3=s_0s_1(\alpha_2)
=s_0(\alpha_1+\alpha_2)
=2\alpha_0+\alpha_1+\alpha_2.
\]
Hence, for $m\geq 0$,
\[
\beta_{3m+1}=c^m(\alpha_0),
\]
\[
\beta_{3m+2}=c^m(\alpha_0+\alpha_1),
\]
and
\[
\beta_{3m+3}=c^m(2\alpha_0+\alpha_1+\alpha_2).
\]

We now claim the following formulae. If $m=2h$, then
\[
\beta_{3m+1}=(3h+1)\alpha_0+3h\alpha_1+3h\alpha_2,
\]
\[
\beta_{3m+2}=(3h+1)\alpha_0+(3h+1)\alpha_1+3h\alpha_2,
\]
and
\[
\beta_{3m+3}=(3h+2)\alpha_0+(3h+1)\alpha_1+(3h+1)\alpha_2.
\]
If $m=2h+1$, then
\[
\beta_{3m+1}=(3h+2)\alpha_0+(3h+2)\alpha_1+(3h+1)\alpha_2,
\]
\[
\beta_{3m+2}=(3h+3)\alpha_0+(3h+2)\alpha_1+(3h+2)\alpha_2,
\]
and
\[
\beta_{3m+3}=(3h+3)\alpha_0+(3h+3)\alpha_1+(3h+2)\alpha_2.
\]

We prove these formulae by induction on $h$. For $h=0$, the even case $m=0$ is exactly the computation of $\beta_1,\beta_2,\beta_3$ above. The odd case $m=1$ is obtained by applying $c$ once:
\[
c(\alpha_0)=2\alpha_0+2\alpha_1+\alpha_2,
\]
\[
c(\alpha_0+\alpha_1)=3\alpha_0+2\alpha_1+2\alpha_2,
\]
and
\[
c(2\alpha_0+\alpha_1+\alpha_2)=3\alpha_0+3\alpha_1+2\alpha_2.
\]
Thus the formulae hold for $h=0$.

For the inductive step, suppose the formulae hold for some $h\geq 0$. 
We check that applying $c^2=(s_0s_1s_2)^2$ to any of the displayed roots, using the relations
\[c(\alpha_0)=2\alpha_0+2\alpha_1+\alpha_2,\quad c(\alpha_1)=\alpha_0+\alpha_2,\quad c(\alpha_2)=-2\alpha_0-\alpha_1-\alpha_2,\] give the formulae with $h$ replaced by $h+1$. For example, if $m=2h$,
\begin{align*}c^2(\beta_{3m+1})&=c^2((3h+1)\alpha_0+3h\alpha_1+3h\alpha_2)\\&=c((3h+2)\alpha_0+(3h+2)\alpha_1+(3h+1)\alpha_2)\\&= (3h+4)\alpha_0+(3h+3)\alpha_1+(3h+3)\alpha_2.\end{align*}

Therefore, all crossed roots have nonnegative coefficients, and
so every crossed root is positive. 
\end{proof}

\begin{lemma}
Let
\[
w=\gamma_0(x)\delta_{x+1}(y)=s_0s_1s_2\cdots s_{x-1} s_{x+1}s_{x}\cdots,
\]
where the subscripts are read modulo $3$. Let
\[
\beta_k=s_{i_1}\cdots s_{i_{k-1}}(\alpha_{i_k})
\]
be the $k$-th crossed root of the whole word $w$. Then every crossed root is positive.
\end{lemma}

\begin{proof}
    For $k\leq x$, by Lemma \ref{beforeturn}, we have that $\beta_k$ is a positive root. Now we can obtain formulae for the rest of the roots, as in the proof above. For instance, when $x\equiv 0\pmod 6$, if $k=x+2r-1$, then
\[
\beta_{k}
=
(r-1)\alpha_0+r\alpha_1+(r-1)\alpha_2,
\]
and if $k=x+2r$, then
\[
\beta_{k}
=
(r+\dfrac{x}{2})\alpha_0+(r+\dfrac{x}{2})\alpha_1+(r+\dfrac{x}{2}-1)\alpha_2.
\]
These are positive roots. The other roots have similar formulae, and are proved by induction. 
\end{proof}

\begin{corollary}
    Any word of the form $\gamma_i(x)\delta_{x+i+1}(y)$ is reduced. 
\end{corollary}

\begin{proof}
    We can assume, without loss of generality, that $i=0$. Then, by the lemmas above, if $w=\gamma_i(x)\delta_{x+i+1}(y)$, and $\beta_k$ is the $k$-th crossed root of $w$, then $\beta_k$ is positive for all $k$. So, by the root criterion for reduced words \cite[Chapter 5.4]{HUMPHREYS2}, $w$ must be a reduced word. 
\end{proof}

We now show that our algorithm outputs the correct coordinates for the set of words in this form. 

\begin{proposition}
    
\label{correctcoords}
    The \hyperref[algorithm]{Words-to-Coordinates Algorithm} outputs the correct region and coordinates for the reduced words 
    \[s_0s_1s_2\cdots s_{\nu-1}s_{\nu+1}s_{\nu+2}s_{\nu}\cdots=\gamma_0(\nu)\delta_{\nu+1}(\mu),\]
    \[s_1s_2s_0s_1s_2s_0\cdots s_{\nu}s_{\nu+2}s_{\nu+1}s_{\nu}s_{\nu+2}\cdots=\gamma_1(\nu)\delta_{\nu+2}(\mu),\]
    \[s_2s_0s_1s_2s_0\cdots s_{\nu+1}s_{\nu}s_{\nu+2}s_{\nu+1}s_{\nu}\cdots=\gamma_2(\nu)\delta_{\nu}(\mu),\]
    with $\nu>0$. 
\end{proposition} 

Before proving this lemma, we first state some basic Euclidean geometry facts of this tiling. 

\begin{lemma} \label{triangles}
    All trapeziums formed from equilateral triangles have angles $\pi/3,\pi/3,2\pi/3,$ and $2\pi/3$. Given any trapezium with angles $\pi/3,\pi/3,2\pi/3,$ and $2\pi/3$, 
\begin{center}

\begin{tikzpicture}[scale=0.8]
    \coordinate (A) at (0, 0);
    \coordinate (B) at (6, 0);
    \coordinate (C) at (5, 1.732); 
    \coordinate (D) at (1, 1.732); 

    \draw[line width=0.05cm] 
        (A) -- (B) -- (C) -- (D) -- cycle;
    \tikzset{every node/.style={ font=\large\itshape}}

    \path (A) -- (D) node[midway, left=0.2cm] {w};
    \path (D) -- (C) node[midway, above=0.1cm] {x};
    \path (C) -- (B) node[midway, right=0.2cm] {y};
    \path (A) -- (B) node[midway, below=0.2cm] {z};
\end{tikzpicture}
\end{center}
with lengths $x,y,w,z$, we have
    \[w=y, \text{ and } z=x+w=x+y.\]
    Furthermore, all triangles formed from equilateral triangles are equilateral.
\end{lemma}

\begin{proof}\textit{(of Proposition \ref{correctcoords})}
    We show that the algorithm outputs the correct coordinates for these words. Let us show that, if the word $w$ is of the form 
    \[\gamma_0(\nu)\delta_{\nu+1}(\mu)=s_0s_1s_2\hdots s_is_{i+2}s_{i+1}s_i\hdots,\]
    then our algorithm outputs the correct values of $n$ and $m$. These words cover the $s_0$ region, so we need to show that the algorithm outputs the $s_0$-region. If $\nu=1$, then the whole word is decreasing, so $w=\delta_0(\mu+1)$. Here, the algorithm outputs that the element lies in the $s_0$-region. Otherwise, the largest starting decreasing subword is $s_0$. So the algorithm correctly identifies that this element is in the $s_0$-region. 
    
    Now, let us assume that $\mu=0$. So our word ends in an increasing subword. Then we are moving in a straight line in the $\beta$-tunnel containing the origin. If $\nu$ is even, the algorithm outputs the coordinates $(\nu,0)$. If $\nu$ is odd, the algorithm outputs the coordinates $(\nu+1,-1)$. Both of these are the correct coordinates for the corresponding group element. 
    
    Now we can assume that $\mu\neq 0$. So our word does not end in an increasing subword. Now our algorithm outputs, for $x$, the number of elements in the beginning increasing subword, and outputs, for $y$, the number of elements in the ending decreasing subword. When $x$ is even, the algorithm gives $(n,m)=(x,y)$. This is exactly the definition of the coordinate system.
Now if $x$ is odd, our algorithm gives $(n,m)=(x+y+1-\epsilon_y,-y-1)$. So we need to show that our word $s_0s_1s_2\hdots s_is_{i+2}s_{i+1}s_i\hdots$ has coordinates $(x+y+1-\epsilon_y,-y-1)$. First assume that $y$ is even. We consider the following picture:
\begin{center}
\resizebox{0.5\textwidth}{!}{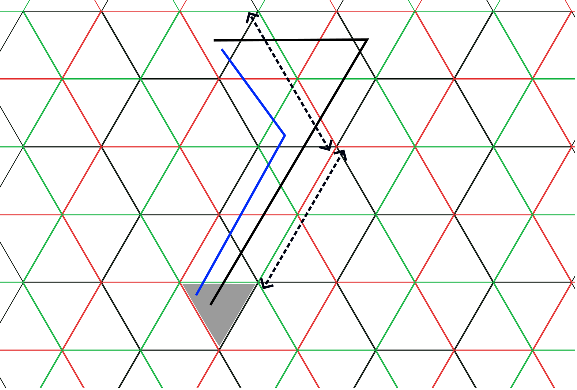}

\end{center}

We can clearly see here that $n=x+y+1$ and $m=-y-1$. If now $y$ is odd, going back one alcove, to $\nu=x$ and $\mu=y-1$ gives us the case above. So then this previous alcove has coordinates $(x+y,-y-2)$. Then to move to the alcove with $\nu=x$ and $\mu=y$, the first coordinate stays the same, and the second coordinate increases by one. So the coordinates are $(n,m)=(x+y,-y-1)$. So overall, for the $s_0$-region, the coordinates are
\[(n,m)=(x+y+1-\epsilon_y,-y-1).\]

The other words and regions follow from a very similar geometric argument.
\end{proof}

\begin{lemma}\label{cover}
Every non-identity element $w\in \aw$ can be written as one of
\[
\gamma_0(\nu)\delta_{\nu+1}(\mu),\qquad
\gamma_1(\nu)\delta_{\nu+2}(\mu),\qquad
\gamma_2(\nu)\delta_{\nu}(\mu),
\]
with $\nu>0$ and $\mu\geq 0$.
\end{lemma}
\begin{proof}
    Let 
    \[S=\{\gamma_0(\nu)\delta_{\nu+1}(\mu), \gamma_1(\nu)\delta_{\nu+2}(\mu), \gamma_2(\nu)\delta_{\nu}(\mu)\mid \nu>0, \mu\geq 0\}.\]
    Every nonidentity element of $W$ can be represented as a coordinate $(n,m)$, with $n$ even and $n,m$ not both zero. We showed, in Proposition \ref{correctcoords}, that the algorithm outputs the correct coordinates for elements in $S$.  So we just need to show that the map defined by the algorithm on $S$ to the set of $(n,m)$, $n,m\in \mathbb{Z}$, $n$ even, $n,m$ not both zero is surjective. 

    So consider an arbitrary $(n,m)$ with $n$ even and $n,m$ not both zero. 
    
    Case 1: $n>0$ and $m\geq 0$.

    Take $\gamma_0(n)\delta_{n+1}(m)\in S$. Then our algorithm outputs $j=0$, and $u=n$. Note that, even if $m=0$, our expression never ends with an odd increasing sequence, so $x=n$, $y=m$. Now $x$ is even, and so our algorithm gives the coordinates $(x,y)=(n,m)$.

    Case 2: $n>0$ and $-n\leq m<0$. 

    Note that $n+m-\epsilon_m+1>0$, and $-m-1\geq 0$. First suppose that $-m-1=0$, so $m=-1$. Here, take $\gamma_0(n-1)$. Then $j=1$ and $u=n-1$, and $n-1$ is odd so $x=n$, $y=-1$. Now $x$ is even, so the coordinates are $(x,y)=(n,-1)=(n,m)$. Now suppose that $-m-1>0$. Here, take $\gamma_0(n+m-\epsilon_m+1)\delta_{n+m-\epsilon_m+1}(-m-1)\in S$. Now $j=0$ and $x=n+m-\epsilon_m+1$ and $y=-m-1$. So $x$ is odd, and we get the coordinates $(x+y+1-\epsilon_y,-y-1)=(n+m-\epsilon_m+1+(-m-1)+1-\epsilon_{-m-1},-(-m-1)-1)=(n,m)$.

    Case 3: $n>0$ and $m<-n$. Take $\gamma_2(-m-n+\epsilon_m)\delta_{-m-n+\epsilon_m}(n-\epsilon_m)$ to give coordinates $(n,m)$.

    Case 4: $n\leq 0$, $0<m\leq -n$. Take $\gamma_1(m+\epsilon_m)\delta_{m+\epsilon_m+2}(-n-m)$ to give coordinates $(n,m)$. 

    Case 5:  $n\leq 0$, $m> -n$. Take $\gamma_1(1-n)\delta_{1-n}(m+n-1)$.

    Case 6: $n=0$, $m< 0$. Take $\gamma_2(-m)$.

    Case 7: $n< 0$, $m\leq 0$. Take $\gamma_2(-m+1-\epsilon_m)\delta _{m+1-\epsilon_m}(-n)$.   
\end{proof}

This concludes our proof of the following theorem,  as all reduced words representing the same group element only differ by a series of braid relations. 

\begin{theorem}
    The \hyperref[algorithm]{Words-to-Coordinates Algorithm} is well-defined and correct, i.e.\ for any element $w\in \aw$, and any reduced word $s_{i_1}\cdots s_{i_k}$ representing $w$, the algorithm outputs the correct coordinates. 
\end{theorem}

\section{Exploring the geometry of $\aw$}\label{sect5}

We now use this coordinate system to explore the structure of $\aw$. We see how the tunnels are completely defined by the coordinates. We also see that the coordinates distinguish the panel arrangement for each alcove. Lastly, we introduce the new notation of channels, and again prove that these can be defined using the coordinate system. 

\subsection{Tunnel inclusion criteria}
We can now use our algorithm to completely describe the tunnels using this notation. By construction, the $(\alpha+\beta)$-tunnels are defined by fixing the first coordinate.

Now the other two sets of tunnels are defined by fixing $m+\epsilon_m$ and $n+m-\epsilon_m$ respectively, where $\epsilon_m=1$ if $m$ is odd and $\epsilon_m=0$ if $m$ is even.

\begin{lemma}\label{betatunnels}
    Two alcoves $(n_1,m_1)$ and $(n_2,m_2)$ lie in the same $\beta$-tunnel if and only if 
    \[m_1+\epsilon_{m_1}=m_2+\epsilon_{m_2}.\]
\end{lemma}

\begin{proof}
    We first prove that any alcove $(n,m)$ lying in the $\beta$-tunnel containing the identity satisfies 
    \[m+\epsilon_m=0.\]
    By Lemma \ref{tunnelsbeta}, alcoves in this tunnel have minimal galleries of type $\gamma_0(\nu)$ or $\delta_2(\nu)$. Computing the coordinates of this alcove using the algorithm gives us the coordinates $(k+\epsilon_k,-\epsilon_k)$ for any alcove in this tunnel. This clearly satisfies $m+\epsilon_m=0$, as $\epsilon_{-\epsilon_k}=\epsilon_k$. Also, if any alcove satisfies $m+\epsilon_m=0$, then $m=0$ or $m=-1$. So the coordinates of the alcove are of the form $(n,0)$ or $(n,-1)$. This translates to galleries of types $\gamma_0(n)$, $\gamma_0(n)s_{n+2}=\gamma_0(n-1)$, $\delta_2(n)$ and $\delta_2(n)s_{1-n}=\delta_2(n+1)$. These galleries are contained in the $\beta$-tunnel containing the identity, and so if an alcove satisfies $m+\epsilon_m=0$, it must be contained in this $\beta$-tunnel. 
    
    The result then follows from the observation that to move to the tunnel on the left we must subtract two from the second coordinate of every alcove, and to move to the tunnel on the right we must add two to the second coordinate of every alcove.
\end{proof}

\begin{lemma}
    Two alcoves $(n_1,m_1)$ and $(n_2,m_2)$ lie in the same $\alpha$-tunnel if and only if 
    \[n_1+m_1-\epsilon_{m_1}=n_2+m_2-\epsilon_{m_2}.\]
\end{lemma}

\begin{proof}
    The proof follows similarly to Lemma \ref{betatunnels}. We show that, for the $\alpha$-tunnel containing the identity, all alcoves satisfy
    \[n+m-\epsilon_{m}=0.\]
    By Lemma \ref{tunnelsalpha}, every alcove in this tunnel has type $\delta_0(\nu)$ or $\gamma_1(\nu)$. This gives us coordinates $(k,-k+1)$, and $(k,-k)$ for $k$ even. This clearly satisfies $n+m-\epsilon_{m}=0$. 
\end{proof}

\subsection{Acceptable folds in an alcove}

We now show that this notation gives an explicit description of the type of alcove that we have, i.e.\ the arrangement of panels in the alcove. We will use this categorisation in later proofs.

\begin{lemma}\label{type} Define the following types of alcoves:

\newcommand{\alcove}[4]{
    \begin{tikzpicture}[scale=0.5, baseline=(current bounding box.center)]
        \ifnum#1=0 
            \draw[#2] (0,0) -- (1,0);
            \draw[#4] (1,0) -- (0.5,0.866);
            \draw[#3] (0.5,0.866) -- (0,0);
        \else 
            \draw[#2] (0,0.866) -- (1,0.866);
            \draw[#3] (1,0.866) -- (0.5,0);
            \draw[#4] (0.5,0) -- (0,0.866);
        \fi
    \end{tikzpicture}
}
\begin{center}

\begin{tabular}{cccccc}
\toprule
\textbf{0} & \textbf{1} & \textbf{2} & \textbf{3} & \textbf{4} & \textbf{5} \\ \midrule
\alcove{1}{black}{red}{green} & 
\alcove{0}{green}{red}{black} & 
\alcove{1}{red}{green}{black} & 
\alcove{0}{black}{green}{red} & 
\alcove{1}{green}{black}{red} & 
\alcove{0}{red}{black}{green} \\ \bottomrule
\end{tabular}
\end{center}

where a black edge represents a panel of type $s_0$, a red edge represents a panel of type $s_1$, and a green edge represents a panel of type $s_2$. 
Consider an alcove with even coordinates $(n,m)$. Let $j$ be the type of $(n,m)$ as defined above. Then the following holds:
\begin{enumerate}
    \item if $n \equiv 0 \pmod 6$, then $j\equiv m \pmod 6$;
    \item if $n \equiv 2 \pmod 6$, then $j\equiv m+4 \pmod 6$;
    \item if $n \equiv 4 \pmod 6$, then $j\equiv m+2 \pmod 6$.
\end{enumerate}
We can write this in a closed form as 
\[j\equiv m+2n \pmod 6.\]

\end{lemma}

\begin{proof}
    We first prove that the type of an alcove only depends on $n$ and $m$ modulo $6$. By Lemmas \ref{tunnels}--\ref{tunnelsalpha}, moving in any tunnel cycles through the types of panels in a 3-cycle. Furthermore, it is easy to see that moving in any tunnel alternates the direction of the alcove in a 2-cycle. So, overall, every sixth element in a tunnel gives the same panel type and direction of alcove. This uniquely identifies the type of the alcove. Adding or subtracting 6 from the second coordinate corresponds exactly to moving six alcoves to the right or left. So, the type of alcove is constant modulo 6 in the second coordinate. Similarly, for the first coordinate, adding or subtracting 6 corresponds to moving up or down (respectively) in a $\beta$-tunnel by 6 alcoves. So, the type of alcove is also constant modulo 6 in the first coordinate.

Now, we only need to show that $j\equiv m+2n\pmod 6$ holds for $n=0,2,4$ and $m=0,1,2,3,4,5$. These cases can be explicitly verified. 

\begin{center}
\resizebox{0.4\textwidth}{!}{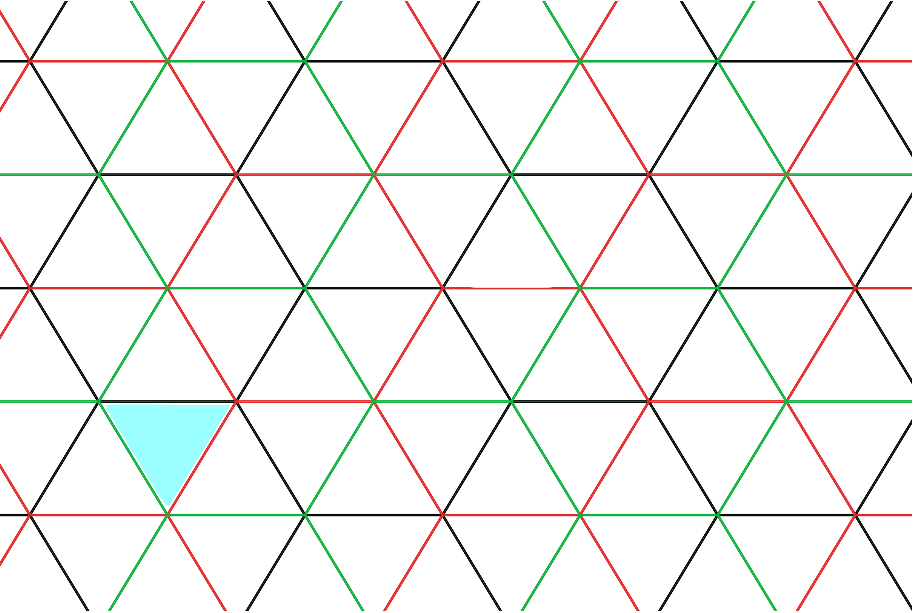}
\end{center}
\end{proof}

\subsection{The channels of $\aw$}\label{defchannels}

Now that we have defined our coordinate system on $\aw$, we define the \ix{channels} of $\aw$. These are the infinite sets of alcoves lying in the same root direction. The next picture illustrates this concept.

\begin{center}
\resizebox{0.4\textwidth}{!}{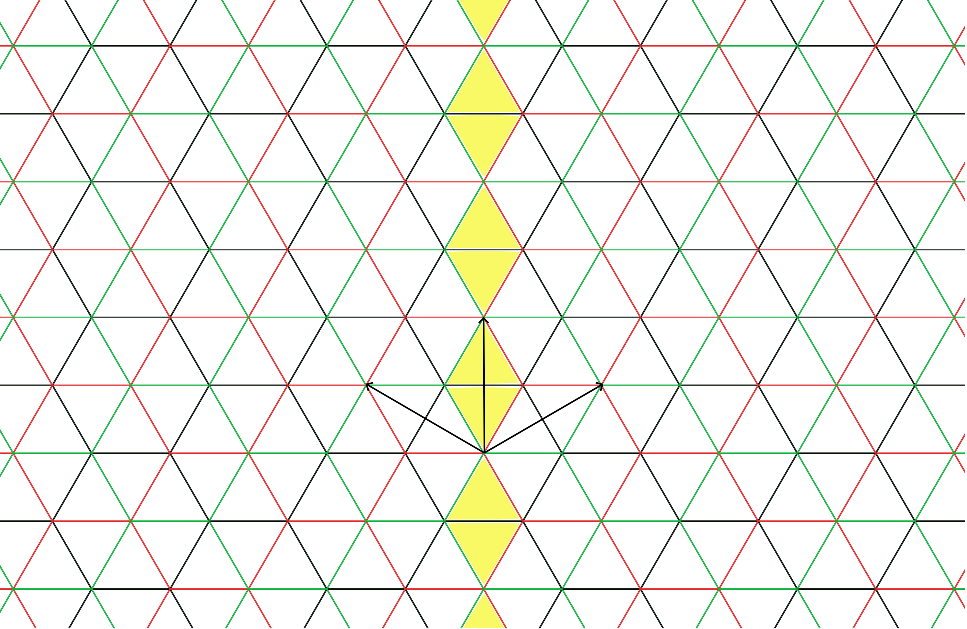}
\end{center}

Now we can shift these channels to the right or left to get new channels, shown here in different colours.

\begin{center}
\resizebox{0.4\textwidth}{!}{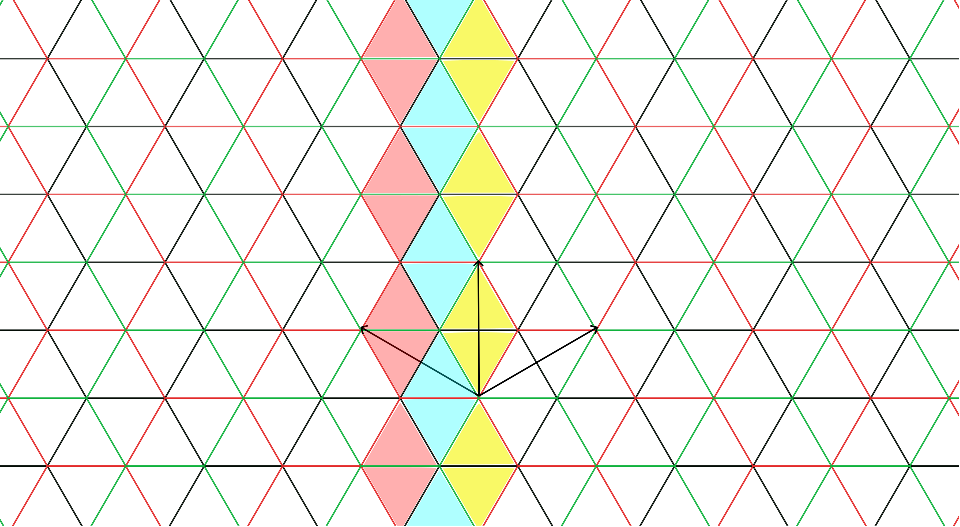}
\end{center}

There are also other channels in the directions of the other roots. The formal definition of a channel is as follows. 

\begin{definition}
    Fix a root $\gamma$. A \ix{$\gamma$-channel} is the orbit of an alcove under reflections in the hyperplanes $H_{\gamma,n}$ for all $n\in \mathbb{Z}$. 
\end{definition}

Here we prove some basic facts about these channels that will help us prove statements about shadows.

Firstly, we show that the panels on each edge of the boundary are consistent. This will be used later to make an inductive argument on the shadow of a gallery. The next lemma follows from Remark \ref{typepreserved}, as the types of panels are preserved under reflections in the hyperplanes. The following lemma does not rely on the specific geometry of $\aw$, and hence applies to any affine Coxeter complex.

\begin{lemma} \label{boundary}
    Each side of the boundary of any channel is made of panels of the same type.
\end{lemma}

Next, we show that the channels and tunnels corresponding to each root intersect uniquely. Again, this proposition does not rely on any geometry specific to $\aw$.

\begin{proposition}\label{once}
    Let $\gamma$ be a root. Consider a $\gamma$-tunnel and a $\gamma$-channel. Then there is exactly one alcove lying in both the tunnel and the channel. 
\end{proposition}

\begin{proof}
    Suppose $H_1=H_{\gamma,t}$ and $H_2=H_{\gamma,t+1}$ are the hyperplanes bounding the tunnel. Let $C$ be the set of alcoves in the $\gamma$-channel. Let $x$ be the element in $C$ such that the number of hyperplanes parallel to $H_1$ and $H_2$ separating $x$ from the interior of the $\gamma$-tunnel is minimal. Suppose that $x$ does not lie in the  $\gamma$-tunnel, and let $v$ be a point in the interior of $x$. Suppose without loss of generality that $l=\langle v,\gamma\rangle <t$. Then the number of hyperplanes separating $x$ from the tunnel is $\lfloor t-l\rfloor+1$. Now let $r$ be the reflection over $H_1$. Then $rv=v-(l-t)\gamma^\vee$, so $\langle rv, \gamma\rangle = l-2(l-t)=2t-l>t$. Now $rx\in C$, but the number of hyperplanes separating $rx$ from the $\gamma$-tunnel is $\lfloor t-l-1\rfloor+1<\lfloor t-l\rfloor+1$, contradicting the choice of $x$ in $C$. So $x$ lies in the tunnel, and so there is at least one element in both the tunnel and the channel.

    Now suppose $x,y$ are two distinct alcoves lying in both the $\gamma$-tunnel and $\gamma$-channel. As $x$ and $y$ are in the same channel, they are both in the orbit of some element with respect to reflections parallel to $H_1$ and $H_2$. So $y=r_n\cdots r_1x$, such that $H_{r_i}$ is parallel to $H_1$ and $H_2$. Now consider $r_n\cdots r_1H_1$. As $x$ and $y$ both lie between $H_1$ and $H_2$, we have that either $H_1$ is fixed setwise, or $r_n\cdots r_1H_1=H_2$. If it is fixed setwise, then it must be fixed pointwise, as all $H_{r_i}$ are parallel to $H_1$. But the only two actions which fix $H_1$ pointwise are the identity and reflection over $H_1$. So either $x=y$ or $y$ is not bounded by $H_1$ and $H_2$. If $r_n\cdots r_1H_1=H_2$, by the same argument $H_2$ is either fixed or $r_n\cdots r_1 H_2=H_1$. But $r_n\cdots r_1H_1=H_2$ so it cannot be fixed. Now $r_n\cdots r_1$ is either a reflection or a translation in the $\gamma$ direction, and it cannot be a translation as the tunnel bounded by $H_1$ and $H_2$ is setwise preserved, as both $x$ and $y$ lie in it. So $r_n\cdots r_1$ is a reflection switching $H_1$ and $H_2$. So $H_{r_n\cdots r_1}$ must be a parallel hyperplane separating $H_1$ and $H_2$, which is a contradiction as $H_1$ and $H_2$ are adjacent. 
\end{proof}

We now use our coordinate system to explicitly define these channels. In particular, we would like to know when two alcoves lie in a common channel. The next propositions and corollaries tell us when this occurs.

\begin{proposition} \label{alphaplusbeta}
    Given a fixed $k\in \mathbb{Z}$, an alcove in $\aw$ with even coordinates $(n,m)$ lies in the same $(\alpha+\beta)$-channel as $(0,k)$ if and only if it satisfies 
    \[n+2m=2k.\]
\end{proposition}
To prove this statement, we first show that the formula holds for an alcove in the $(\alpha+\beta)$-channel that contains the identity alcove $(0,0)$. We then use the fact that moving to the channel on the right fixes the first coordinate and increases the second coordinate by one. Similarly, moving to the channel on the left fixes the first coordinate and decreases the second coordinate by one. So if the formula holds for the $(\alpha+\beta)$-channel containing the identity, then it must hold for all $(\alpha+\beta)$-channels.

We then prove that, given an alcove satisfying the equation, it must lie in the $(\alpha+\beta)$-channel containing $(0,k)$, by arguing that the channels meet each $(\alpha+\beta)$-tunnel exactly once, by Proposition \ref{once}. Given that we have shown that the element in the same channel satisfies this equation, this must be the unique element.

\begin{proof}
    $\Rightarrow$: We start by showing the equation holds for an alcove in the $(\alpha+\beta)$-channel that contains the identity. We do this by explicitly calculating the coordinates of elements in this channel.

    Consider two alcoves $x,y$ in this channel which are adjacent. By Lemma \ref{boundary}, the shared panel must have type $s_0$. Furthermore, to move between two alcoves in the channel that only share a vertex, we multiply by $s_1s_2s_1$.

\begin{center}
\resizebox{0.4\textwidth}{!}{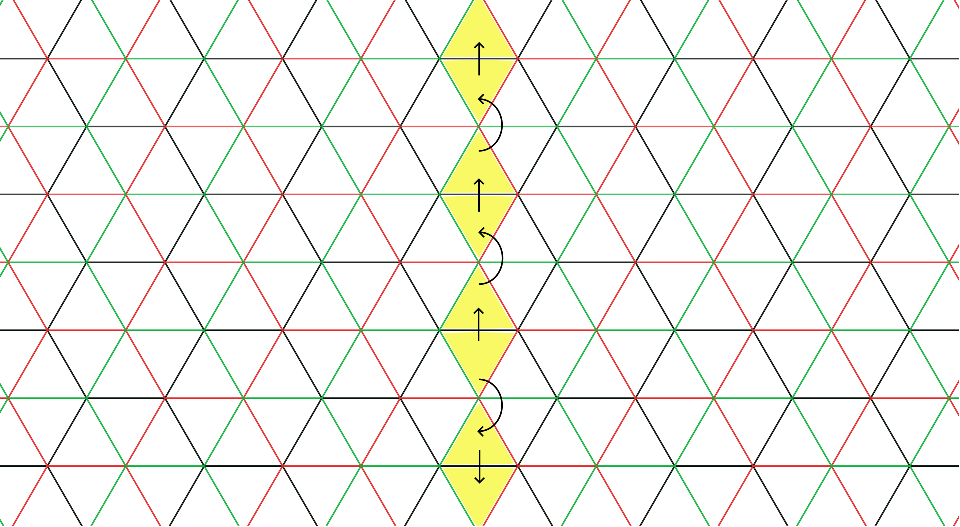}
\end{center}

    So now each element in this channel is one of four types:
    \begin{enumerate}
        \item $s_0s_1s_2s_1s_0s_1s_2s_1s_0\hdots s_1s_2s_1, $
        \item $s_0s_1s_2s_1s_0s_1s_2s_1s_0\hdots s_1s_2s_1s_0, $
        \item $s_1s_2s_1s_0s_1s_2s_1s_0\hdots s_1s_2s_1, $
        \item $s_1s_2s_1s_0s_1s_2s_1s_0\hdots s_1s_2s_1s_0. $
    \end{enumerate}
    We can now use the algorithm to output the coordinates of any alcove in this channel, and show that it satisfies the equation. We show that the equation holds for the first case. The others are very similar.
    
     So let us assume that $w=s_0s_1s_2s_1s_0s_1s_2s_1s_0\hdots s_1s_2s_1. $ This lies in the $s_0$-region. So, using the \hyperref[algorithm]{Words-to-Coordinates Algorithm}, we get 
    \[w=\underbrace{s_0s_1s_2}_\uparrow\underbrace{s_1s_0}_\downarrow \underbrace{s_1s_2}_\uparrow\underbrace{s_1s_0}_\downarrow\hdots \underbrace{s_1s_2}_\uparrow\underbrace{s_1.}_\downarrow\]

    So $x_w=2t+1$ and $y_w=2t-1$ for some $t$. Then the coordinates from the algorithm are $(n_w,m_w)=(4t,-2t)$. This clearly satisfies the equation. The other cases give us coordinates of the form $(4t+2,-2t-1)$, $(-4t+2,2t-1)$ and $(-4t,2t)$, respectively. These all clearly satisfy the equation.

    So, overall, all alcoves in the $(\alpha+\beta)$-channel containing the identity satisfy the equation. Now we note, by construction, that if $(n,m)$ lies in the channel containing the $(0,k)$ alcove, then $(n,m-k)$ lies in the channel containing the identity alcove. So, therefore,
    \[n+2(m-k)=0,\]
    and hence,
    \[n+2m=2k.\]

    $\Leftarrow$: Now we show that, if an alcove $(n,m)$ satisfies the equation, it must lie in the channel containing $(0,k)$. So suppose it satisfies the equation.  Now, by Proposition \ref{once}, the $(\alpha+\beta)$-channel meets the $(\alpha+\beta)$-tunnel in exactly one alcove. The $(\alpha+\beta)$-tunnels are defined by fixing the value of the first coordinate.  Fixing $n$, there is exactly one solution to the equation $n+2m=2k$ for a fixed $k$. We have shown above that the alcove in the $(\alpha+\beta)$-channel containing $(0,k)$ satisfies this equation, so this must be $(n,m)$. 
\end{proof}

We can now easily conclude a corollary that tells us when two alcoves lie in a common channel. 

\begin{corollary}
    Two alcoves in $\aw$, with even coordinates $(n_1,m_1)$ and $(n_2,m_2)$, lie in a common $(\alpha+\beta)$-channel if and only if
    \[n_1-n_2=2(m_2-m_1).\]
\end{corollary}

Now we have very similar propositions and corollaries for the other two types of channels. In these two cases, to move from one channel to the next we must add or subtract two from the second coordinate. Also, for the $\beta$-channels, the uniqueness argument is a little more subtle - for a fixed $n$ there are two values of $m$ that satisfy the equation. We note here that we are distinguishing the channels by the elements $(0,2k)$. In the case of the $\beta$-channels, $(0,2k)$ and $(0,2k+1)$ lie in the same channel. For the $\alpha$-channels, $(0,2k)$ and $(0,2k-1)$ both lie in the same channel. 

\begin{proposition}

    Given a fixed $k\in \mathbb{Z}$, an alcove in $\aw$ with even coordinates $(n,m)$ lies in the same $\beta$-channel as $(0,2k)$ if and only if it satisfies 
    \[2n+m-\epsilon_m=2k.\]
\end{proposition}

\begin{proof}
    $\Rightarrow$: Following a very similar argument to Proposition \ref{alphaplusbeta}, we first show that the equation holds for all alcoves in the $\beta$-channel containing the identity. We do this by again explicitly writing down the group elements contained in the channel. We then run the algorithm and conclude that all elements of this channel are of the form $(2t,-4t)$ or $(2t,-4t+1)$ for some $t$. Both cases clearly satisfy the equation, and so the equation holds for elements in the $\beta$-channel containing $(0,0)$. 

    Now we note that, if $(n,m)$ is in the channel containing $(0,2k)$, then $(n,m-2k)$ is in the channel containing $(0,0)$. So
    \[2n+(m-2k)-\epsilon_{m-2k}=0.\]
    Now $\epsilon_{m-2k}=\epsilon_m$, so we get
    \[2n+m-\epsilon_m=2k.\]

    $\Leftarrow$: By Proposition \ref{once}, a $\beta$-channel and a $\beta$-tunnel meet each other exactly once. We show that, in each $\beta$-tunnel, the equation has exactly one solution. 

    Our $\beta$-tunnels are defined by $m+\epsilon_m=2r$, for some $r$. Solving this simultaneously with the equation $2n+m-\epsilon_m=2k$ gives us the solution 
    \[n=r+\epsilon_m.\]
    We note that there are exactly two solutions to $m+\epsilon_m=2r$, one odd and one even. But our first coordinate $n$ must be even. So, for a fixed $r$, we can only take the one value of $m$ such that $m+\epsilon_m=2r$ and $r+\epsilon_m$ is even. This then fixes $n$. Therefore, for fixed $r$, and so for each tunnel, there is exactly one solution to the equation. We have shown that the element in the channel solves the equation, so this must be the unique one. 
\end{proof}

Again, we can now state a corollary for when two alcoves lie in the same $\beta$-channel. 

\begin{corollary}
    Two alcoves in $\aw$, with even coordinates $(n_1,m_1)$ and $(n_2,m_2)$, lie in a common $\beta$-channel if and only if
    \[2(n_1-n_2)=(m_2-\epsilon_{m_2})-(m_1-\epsilon_{m_1}).\]
\end{corollary}
    
Lastly, we have a proposition and corollary for the $\alpha$-channels.

\begin{proposition}\label{alpha}
    Given a fixed $k\in \mathbb{Z}$, an alcove in $\aw$ with even coordinates $(n,m)$ lies in the same $\alpha$-channel as $(0,2k)$ if and only if it satisfies 
    \[m-n+\epsilon_m=2k.\]
\end{proposition}

\begin{proof}
    $\Rightarrow$: Very similarly to the last two proofs, we conclude that the elements in the $\alpha$-channel containing $(0,0)$ have coordinates $(2t,2t)$ or $(2t,2t-1)$, which both satisfy the equation. Then, to move to the next channel, we add two to the second coordinate. So, if $(n,m)$ lies in the channel containing $(0,2k)$, $(n,m-2k)$ lies in the channel containing $(0,0)$. Hence, the alcove $(n,m)$ satisfies the equation. 

    $\Leftarrow$: By definition, an $\alpha$-tunnel and $\alpha$-channel meet in exactly one alcove. Each $\alpha$-tunnel is defined by the equation $n+m-\epsilon_m=2r$ for some $r$. Now simultaneously solving this with the equation $m-n+\epsilon_m=2k$ gives the unique solution
    \[m=r+k, \quad n=r-k+\epsilon_{r+k}.\]
    As we have shown that the alcove lying in the channel satisfies the equation, this must be the unique solution. 
\end{proof}
    
\begin{corollary}
    Two alcoves in $\aw$, with even coordinates $(n_1,m_1)$ and $(n_2,m_2)$, lie in a common $\alpha$-channel if and only if
    \[n_1-n_2=(m_1+\epsilon_{m_1})-(m_2+\epsilon_{m_2}).\]
\end{corollary}

Now we prove a lemma that tells us that the shadow with respect to the trivial orientation has a line of symmetry. Later, we will combine this with the $(\alpha+\beta)$-channels to describe the shadow. We will do this by explicitly stating how many elements in each $(\alpha+\beta)$-channel appear in the shadow, and so, by the line of symmetry, we will completely describe the shadow.

When working with the trivial orientation, we can always assume that our starting alcove is in the $s_0$-region. Otherwise, we can apply an automorphism permuting the generating elements. Here, we take the reduced gallery  \[s_0s_1s_2\hdots s_is_{i+2}s_{i+1}s_i\hdots=\gamma_0(\nu)\delta_{\nu+1}(\mu).\]
By Lemma \ref{cover}, any element in the $s_0$-region is the end alcove of some gallery of this form.

\begin{lemma}\label{symmetry}
    There is a line of symmetry in the shadow of any alcove in the $s_0$-region, with respect to the trivial orientation, about the hyperplane containing the $s_0$-panel of the identity alcove.
\end{lemma}

\begin{proof}
    By construction, the first panel of the gallery taken is the $s_0$-panel of the identity alcove. Call this panel $p$. Now consider any element $w$ in the shadow of this gallery. There are two cases to consider. 
    \begin{enumerate}
        \item The folded gallery to $w$ has a fold at $p$. Then we unfold the gallery at panel $p$ to get the element $s_0w$, which must also be in the shadow.
        \item The folded gallery does not have a fold at $p$. So the word $x$ describing the footprint of this gallery starts with $s_0$. Also, $x$ and $w$ must represent the same element in $W$. Now we can add a new fold at panel $p$ to our folded gallery. This produces the element $s_0x=s_0w$.

    \end{enumerate}
    As $s_0(s_0w)=w$, we conclude that $w$ is in the shadow of a standard gallery if and only if $s_0w$ is in the shadow.

        Now multiplying elements on the left by $s_0$ is the automorphism of $\aw$ that reflects each alcove over the hyperplane containing panel $p$. So there must be a line of symmetry in the shadow across this hyperplane. 
\end{proof}

\section{Calculations of Bruhat intervals}\label{calculations}

Here we use our notation to explicitly calculate shadows in $\aw$ with respect to the trivial orientation. This shadow exactly contains the elements that are smaller than our end alcove of the gallery with respect to the Bruhat ordering, see \cite{MASTERS}. For this orientation, as discussed in the previous section, we can make the assumption that our end alcove lies in the $s_0$-region. We can view this shadow in a number of different ways.

\subsection{The convex hull of points}

Here is a picture of the shadow of the alcove with notation $(6,7)$. The purple coloured alcove is the alcove $(6,7)$. 
\begin{center}
\includegraphics[scale=0.1]{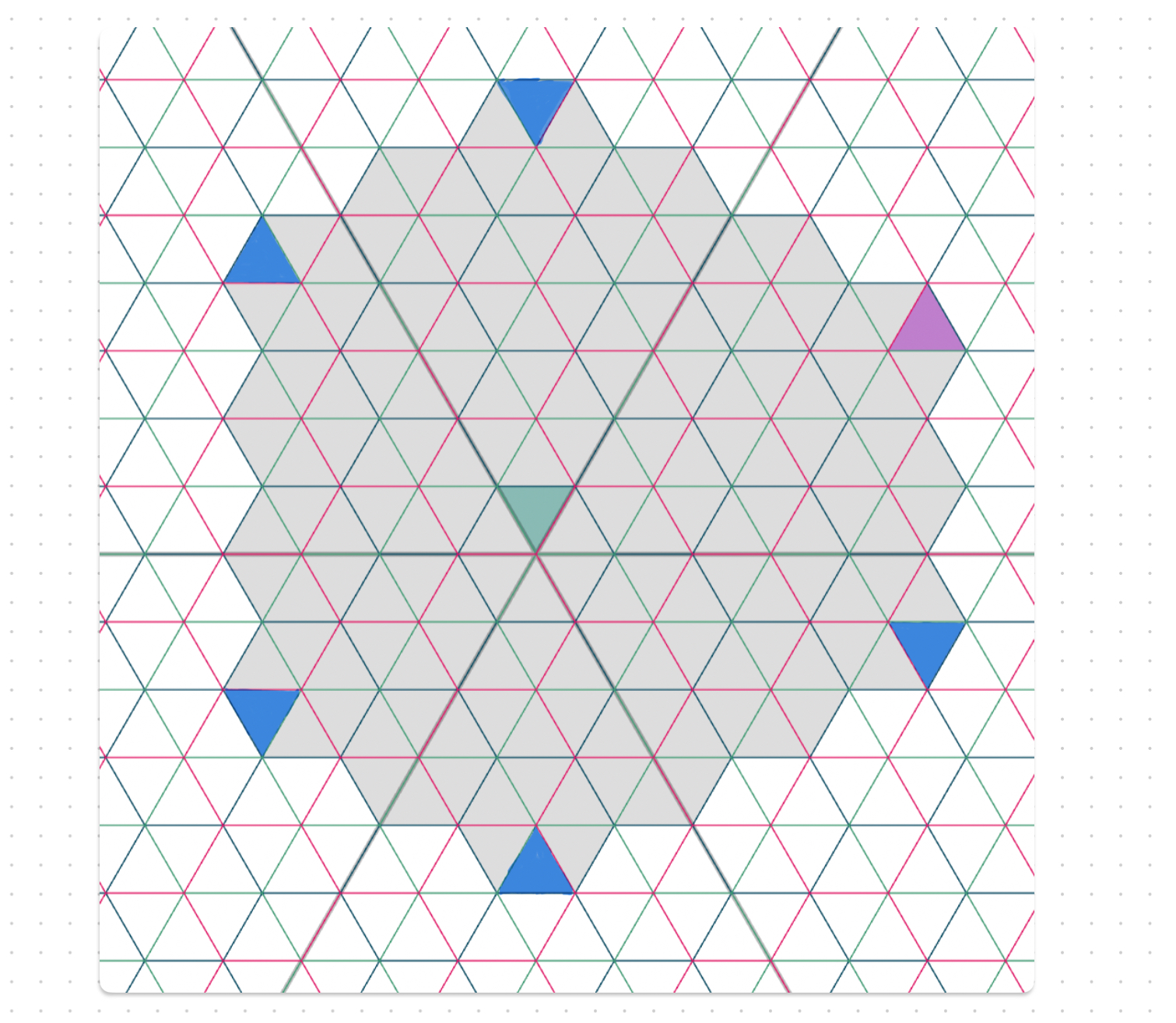}
\end{center}

We can see here that the shadow is the convex hull of the 6 highlighted alcoves. For each alcove $(n,m)$, these boundary alcoves are completely described by $n$ and $m$ by the following proposition. Here we allow $n$ to be odd or even, and for each alcove we choose the coordinates that gives us a nonnegative second coordinate. This is achieved by using the involution from Section \ref{involution}. 
\begin{proposition}\label{convex}
    Let $w=(n,m)$, with $n\geq 2$ and $m\geq 1$. Then the shadow of $w$ with respect to the trivial orientation is the convex hull of the following alcoves:
    \begin{align*}
    &(n,m),\\
    &(3-n+\epsilon_n,n+m-3+\epsilon_n),\\
    &(3-n-m-\epsilon_n+\epsilon_m,n-2-\epsilon_n-\epsilon_m),\\
    &(n+1,m-2),\\
    &(2-n-\epsilon_n,n+m-1-\epsilon_n),\\
    &(2-n-m-\epsilon_n+\epsilon_m,n-\epsilon_n-\epsilon_m),
    \end{align*}
    where $\epsilon_k=1$ if $k$ is odd and $\epsilon_k=0$ if $k$ is even.
\end{proposition}

To prove this pattern, we use induction, and a corollary of Theorem 7.1 in \cite{SHA}, which tells us that
\begin{equation}\tag{$\star$}
    \text{Sh}(xs_i)=\text{Sh}(x)\cup \{us_i\mid u\in \text{Sh}(x)\}.
\end{equation}

So the shadow of an element $xs_i$ is the shadow of $x$, plus any alcoves not in the shadow of $x$ that share a panel of type $i$ with an alcove in the shadow of $x$. 

Now this proposition only holds for $(n,m)$ with $n\geq 2$ and $m\geq 1$. Let us see what occurs in the other, simpler cases. For $(0,0)$, the shadow is trivially the identity alcove. For $(1,0)$, the shadow is the identity alcove, and the alcove $(1,0)$. Now for arbitrary $(n,0)$ and $(1,m)$, the shadow is the convex hull of four alcoves. These two cases are completely symmetrical, as we can permute the generating elements $s_1$ and $s_2$ to map the $(1,m)$ case to $(m+1,0)$. So, before proving Proposition \ref{convex},  we prove the next lemma, which describes the 4 boundary alcoves of $(n,0)$ in terms of $n$.

\begin{lemma} \label{convexlemma}
    Consider an alcove $w=(n,0)$, with $n\geq 2$. Then the shadow of $w$ with respect to the trivial orientation is the convex hull of the following alcoves:
    \begin{align}
    &(n,0),\\
    &(2,1-n),\\
    &(0,2-n),\\
    &(2-n-\epsilon_n,n-1).
    \end{align}

\end{lemma}

To prove this lemma, we make an induction argument. We assume that the lemma holds for the $(n,0)$ alcove, and show that this shadow has a boundary made of channel segments. We then consider when $n$ is even and when $n$ is odd. Let us see an example. Starting with the alcove $(6,0)$, as 6 is even, we need to add the following yellow alcoves to form the shadow of $(7,0)$:  

\begin{center}
\resizebox{0.3\textwidth}{!}{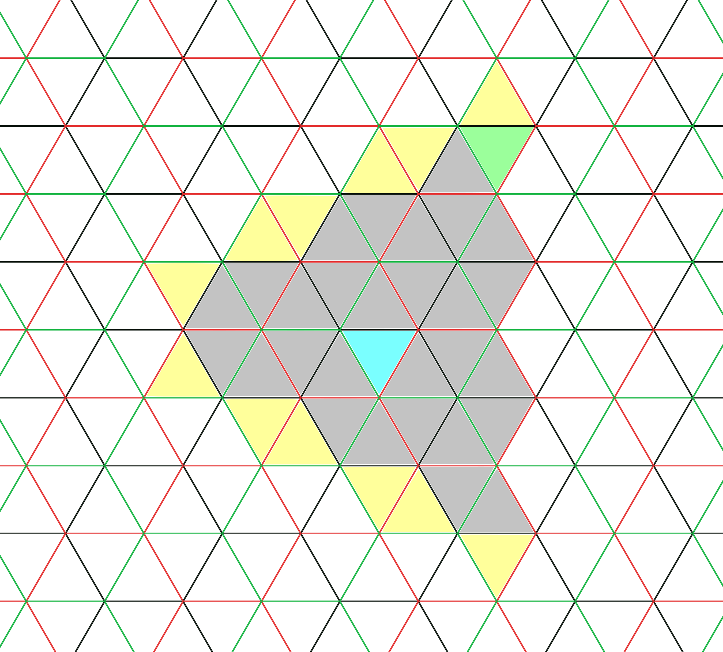}
\end{center}

Now, when we move from $(7,0)$ to $(8,0)$, as 7 is odd, we argue that we need to add these alcoves highlighted in yellow:

\begin{center}
\resizebox{0.3\textwidth}{!}{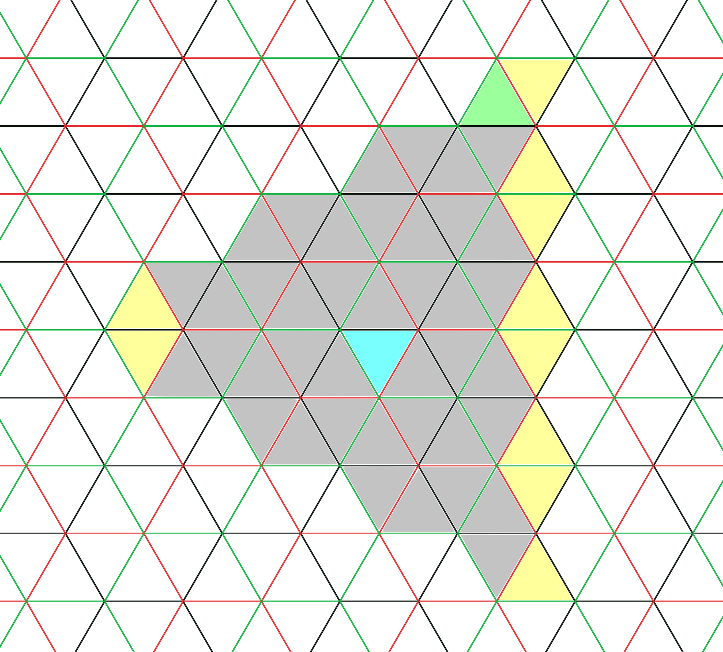}
\end{center}

\begin{proof}
    We proceed by induction. Let us start with the base case of $n=2$. Then the shadow of the alcove $(2,0)$ can be easily computed as the following picture, with the coordinates labelled.

\begin{center}
\resizebox{0.3\textwidth}{!}{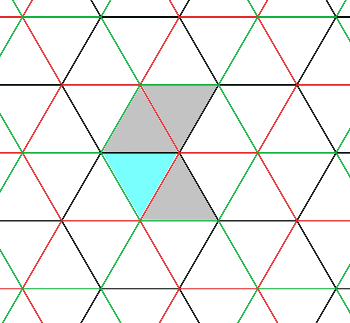}
\end{center}

     Now the set of alcoves given by the lemma is $(2,0), (2,-1), (0,0),$ and $(0,1)$. We clearly obtain the same set as the shadow of the alcove. 

     Now we assume that the lemma holds for the alcove $(n,0)$. We first conclude that the boundary of this shadow is made of segments of channels, by using the Propositions and Corollaries \ref{alphaplusbeta}--\ref{alpha}. We conclude that alcove $(1)$ and alcove $(2)$ lie in a channel, alcoves $(2)$ and $(3)$ lie in a channel, alcoves $(3)$ and $(4)$ lie in a channel, and $(1)$ and $(4)$ lie in a channel.

     For instance, alcoves $(1)$ and $(4)$ have coordinates $(n_1,m_1)=(n+\epsilon_n,-\epsilon_n)$ and $(n_2,m_2)=(2-n-\epsilon_n,n-1)$, respectively. These satisfy the equation
     \[n_1-n_2=2(m_2-m_1),\]
     and so, by Lemma \ref{tunnels}, these lie in the same $(\alpha+\beta)$-channel.

     Let us assume that the lemma is true for the alcove $w=(n,0)$. We show that the lemma must also be true for the alcove $ws_i=(n+1,0)$. Now we split into two cases depending on the parity of $n$.

     Case 1: $n$ is even.

     Then the alcove $(n,0)$ is downwards facing. When we move to $(n+1,0)$, we are crossing over the panel $p$ at the top of the $(n,0)$ alcove. Say this has type $s_i$. Now, as it is at the top of the alcove, it is contained in the boundary of the channel segment containing alcoves (1) and (2). Now, by Lemma \ref{type}, this whole boundary has the same type as panel $p$. So $\Sh{ws_i}$ must contain every alcove that shares a panel with this channel segment. 

    Now we look at the channel segment between alcoves (1) and (4). The panels of the outside boundary of this channel segment cannot have the same type as panel $p$, as it contains another panel of alcove $(n,0)$. 

    Next, let us argue that the boundary of the channel segment between alcoves (2) and (3) has the same type as panel $p$. By Lemma \ref{type}, alcove (2) is an upwards facing alcove. Now the channel containing (1) and (2) is an $\alpha$-channel, and the channel containing (2) and (3) is an $(\alpha+\beta)$-channel. So the panel of alcove (2) in the boundary of the $\alpha$-channel is also the panel in the boundary of the $(\alpha+\beta)$-channel. So the boundary of the channel segment between (2) and (3) has the same type. So any alcove sharing a panel on this boundary must be included in the shadow of $(n+1,0)$. 
    
    Now, by using the line of symmetry of the shadow by Lemma \ref{symmetry}, and by using the fact that panel types are preserved under reflections in hyperplanes, we conclude that the boundary of the channel segment between alcoves (3) and (4) also has the same type of panels as $p$. Hence, any alcove sharing a panel on this boundary must also be included in the shadow of $(n+1,0)$. 

    So we conclude that the shadow of $(n+1,0)$ is the shadow of $(n,0)$ plus the alcoves sharing a panel with the channel segments between (1) and (2), (2) and (3), and (3) and (4).
\begin{center}
\resizebox{0.3\textwidth}{!}{\input{Pictures/use61457.eps_tex}}
\end{center}

Now it is obvious that, moving to the next set of four alcoves, the convex hull is exactly the shadow for $(n,0)$ and these new channel segments.

Case 2: $n$ is odd.

A very similar argument can be made. Now $(n,0)$ is an upwards facing alcove, and the panel $p$ between $(n,0)$ and $(n+1,0)$ lies on the boundary between alcoves (1) and (4). Say this panel has type $s_i$. Then the boundary between alcoves (1) and (4) must only have panels of type $s_i$. 

Now the boundary of the channel segment between alcoves (1) and (2) contains another panel of alcove (1), and so cannot have type $s_i$. Then, by the line of symmetry from Lemma \ref{symmetry}, the boundary of the channel segment between alcoves (3) and (4) also cannot have type $s_i$. So we do not need to add the boundary of these channel segments.

Lastly, we consider the boundary of the channel segment between alcoves (2) and (3). We show that this boundary has type $s_i$, by using Lemma \ref{type} to show that the panel to the left in alcove (2) has type $s_i$.

We claim that, if alcove $(n+1,-1)$ has type $j$, alcove $(2,1-n)$ has type $j+1$. This follows easily from Lemma \ref{type}.

Now, by looking at all possible types of alcove (1), and remembering that alcove (1) is downwards facing, we conclude that the right panel of alcove (1) has the same type as the left panel of alcove (2). Hence, we conclude that we must add the boundary of the channel segment between alcoves (2) and (3). 

So, overall, we are adding the following alcoves:

\begin{center}
\resizebox{0.3\textwidth}{!}{\input{Pictures/use61458.eps_tex}}
\end{center}

Now again, it is obvious that moving to the next set of four alcoves gives the correct convex hull.

\end{proof}

Now Proposition \ref{convex} can be proved, in a very similar way, except this time we have to consider 6 channel segments on the boundary. The proof is omitted for space.

\subsection{Number of elements in the shadow}

In this section, for each alcove, we choose the representation that has nonnegative second coordinate.  Now we see that the shadow of each alcove can be described by specifying the number of elements in each $(\alpha+\beta)$-channel. We use the line of symmetry of the hyperplane defined by the $s_0$-panel of the identity alcove to determine exactly which alcoves are in the shadow. 

We will use Proposition \ref{convex} to calculate the sizes of the channels. For instance, here is the shadow of the alcove defined by $(6,3)$:

\begin{center}
\includegraphics[scale=0.2]{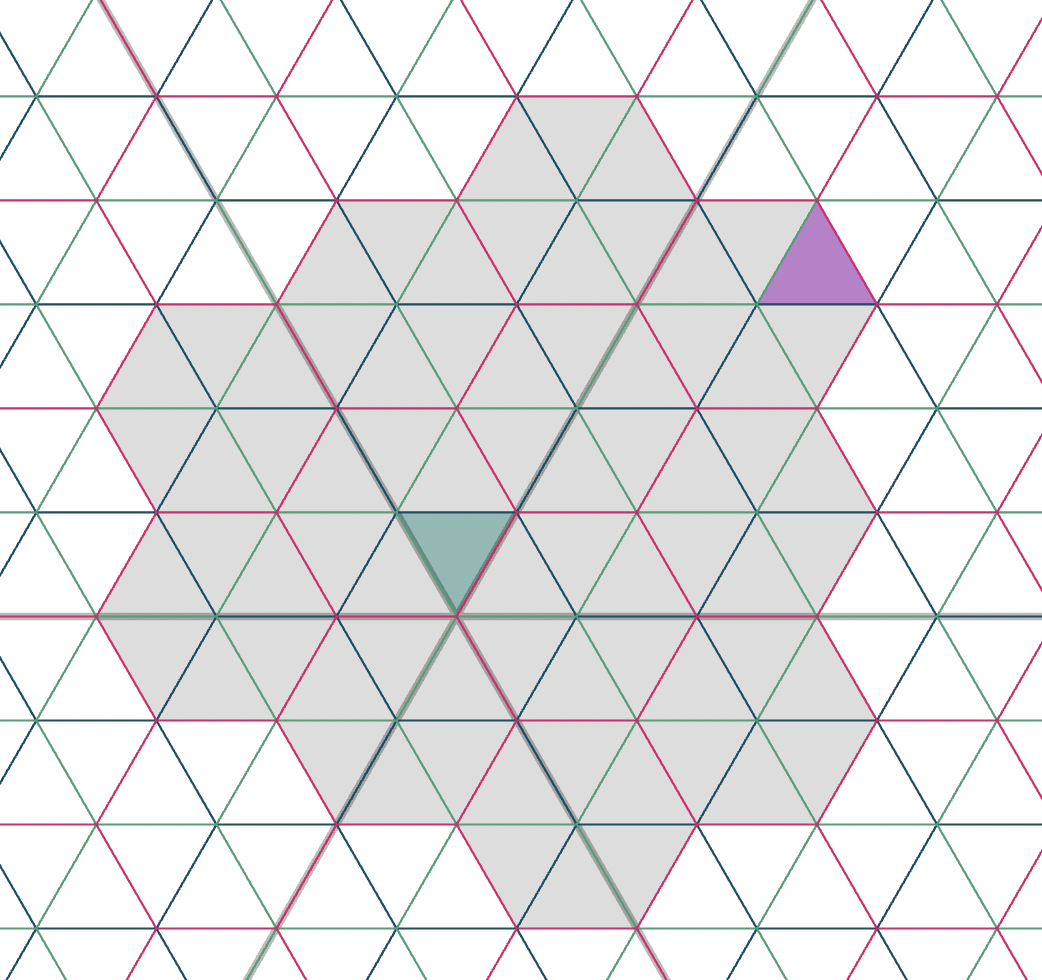}
\end{center}

Now, starting from the channel containing the alcove $(6,3)$ and moving to the left, the number of elements in each $(\alpha+\beta)$-channel is 
\[6,6,6,8,8,8,6,6,6,4,4,4.\]
This pattern seen here holds for every alcove $(n,m)$ with $n$ even and $m$ odd.

A slightly different pattern emerges when $m$ is even. Here is an example:
\begin{center}
\includegraphics[scale=0.18]{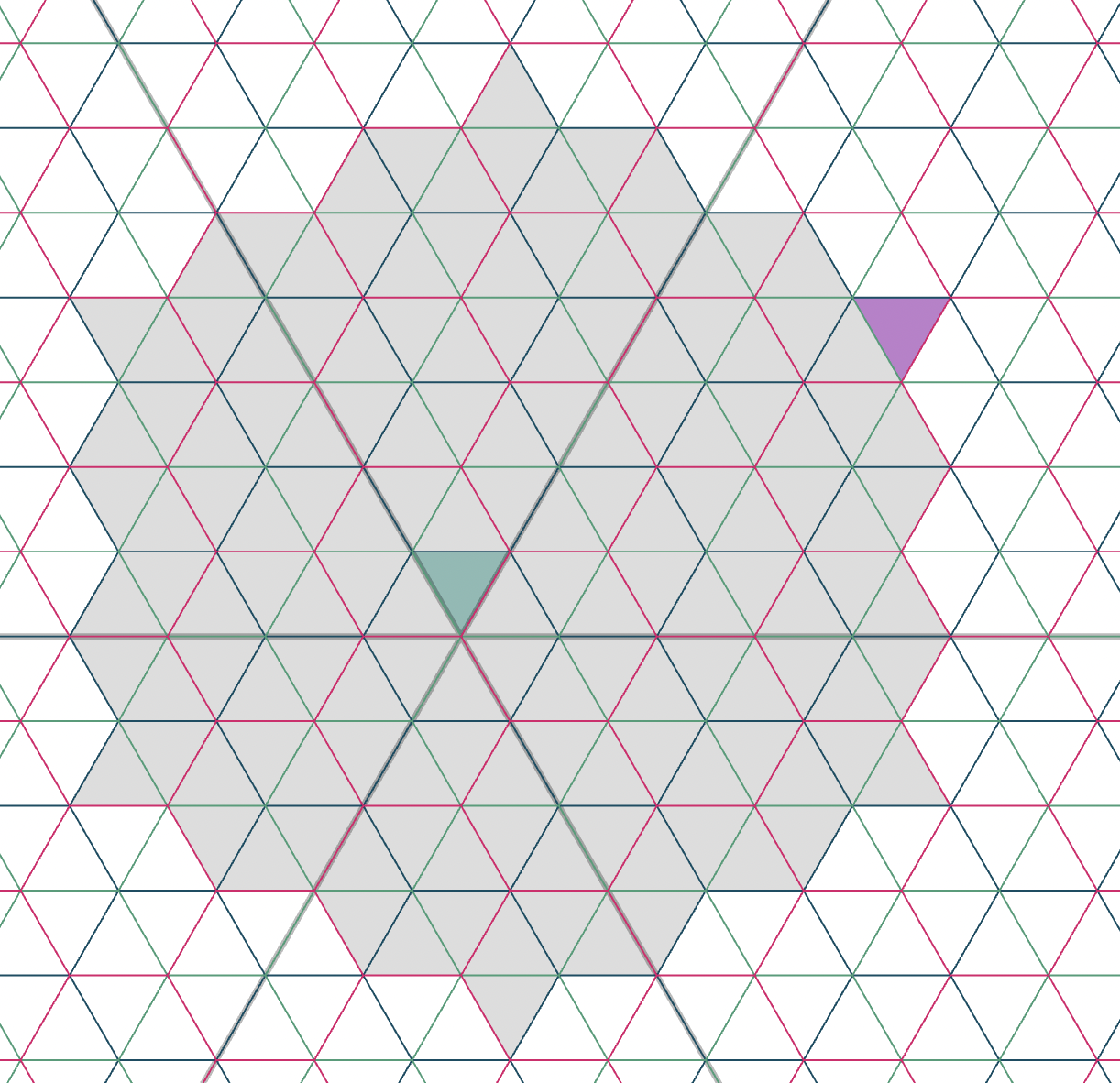}
\end{center}

Now starting from the $(\alpha+\beta)$-channel containing the alcove $(6,6)$, moving to the left the number of elements in each channel is 
\[6,6,8,8,8,10,10,10,12,10,10,10,8,8,8,6,6.\]
This pattern seen here holds for every alcove $(n,m)$ with $n$ and $m$ even.

Very similar patterns occur when $n$ is odd, except that the starting alcove's channel now appears not to the right of any other non-empty channel, but somewhere in the middle. Here is an example, where our starting alcove is $(5,7)$.
\begin{center}
\includegraphics[scale=0.2]{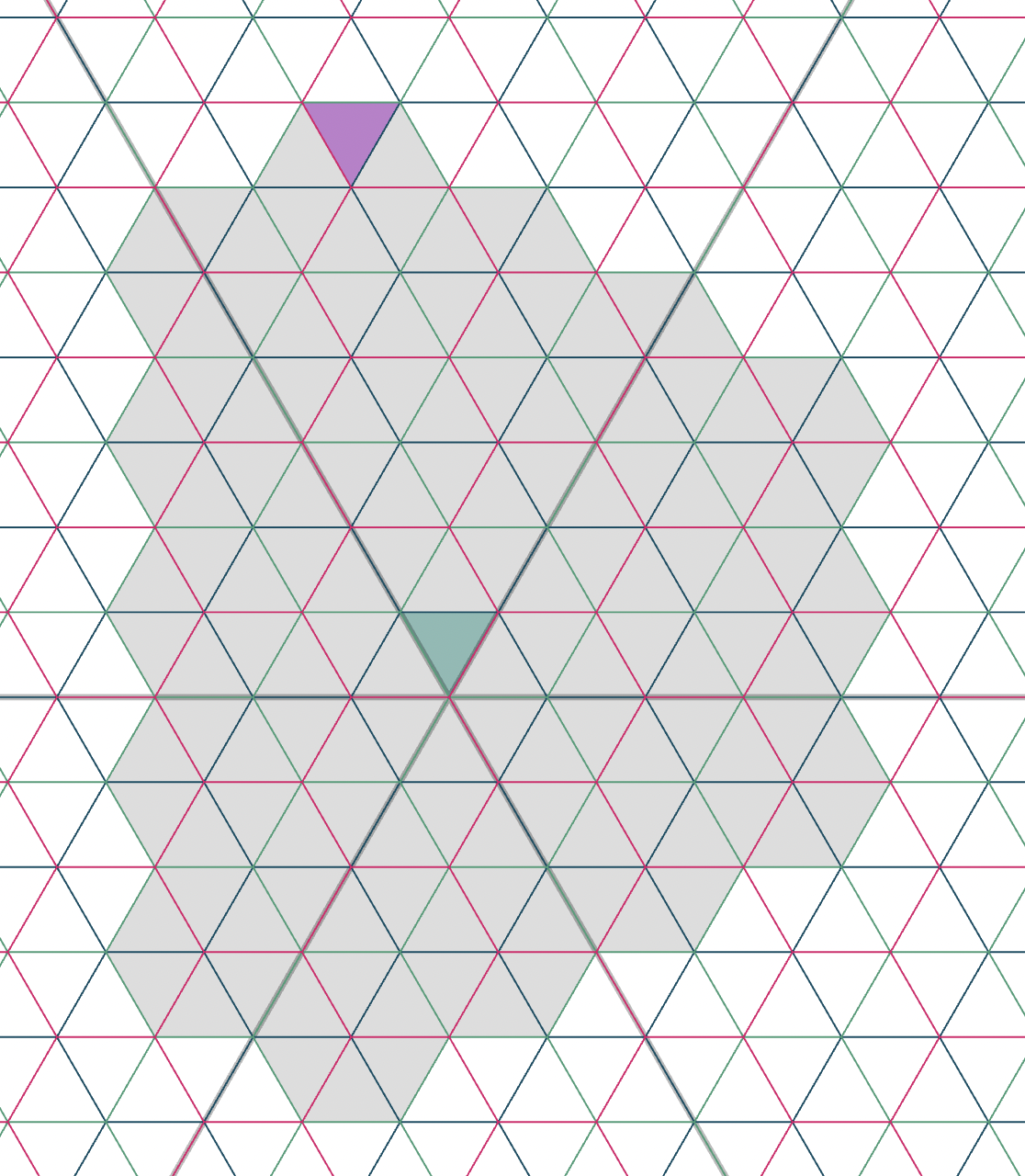}
\end{center}
Here, the size of each channel, starting from the left and denoting the channel containing $(0,0)$ by bold face, is
\[10,10,10,12,12,12,\mathbf{10},10,10,8,8,8,6,6,6.\]

Calculating this string of numbers significantly reduces the computation required to conclude whether a particular alcove lies in the shadow. For instance, say we want to conclude whether the alcove $(-4,6)$ lies in the shadow of $(5,7)$. We note, by Proposition \ref{alphaplusbeta}, $(-4,6)$ lies in the same $(\alpha+\beta)$-channel as the alcove $(0,4)$. Then, by the sequence above, the shadow of $(5,7)$ contains 8 elements in this  $(\alpha+\beta)$-channel. As we have a symmetry in the shadow, these 8 elements lie in the $(\alpha+\beta)$-tunnels defined by $n\in\{-6,-4,-2,0,2,4,6,8\}$. Therefore, $(-4,6)$ lies in the shadow of $(5,7)$.

\begin{proposition}\label{numchannels}
Let $(n,m)$ be an alcove in the $s_0$-region, written with nonnegative second coordinate, and suppose that $n\geq 2$, $m\geq 0$.Consider the shadow of this alcove, and define the sequence of the size of the channels to be the ordered list of the sizes of all nonempty $(\alpha+\beta)$-channels in the shadow.
If $m=0$, and $n$ is even, the sequence of the size of the channels is
\[2^{(3)} ,4^{(3)} ,6^{(3)},\hdots,n-2^{(3)},\mathbf{n}^{(2)},\]
where the bold face represents the channel that contains $(n,m)$, and the superscripts represent the multiplicity. 
If $m=0$, and $n$ is even, the sequence of the size of the channels is
\[2^{(2)} ,4^{(3)} ,6^{(3)},\hdots,n-3^{(3)},\mathbf{n+1}^{(1)}.\]

Now suppose $m>0$. If $n$ and $m$ are even, the sequence of the size of the channels is 
    \[m^{(2)} ,m+2^{(3)} ,m+4^{(3)},\hdots, n+m-2^{(3)},n+m,n+m-2^{(3)},n+m-4^{(3)},\hdots, n+2^{(3)},n,\mathbf{n},\]
    
    If $n$ is even and $m$ is odd, the sequence of the size of the channels is
    \[m+1^{(3)},m+3^{(3)},\hdots,n+m-1^{(3)}, n+m-3^{(3)},\hdots, n,n,\mathbf{n}.\]

    If $n$ is odd and $m$ is even, the sequence of the size of the channels is 
    \[m+2^{(2)} ,m+4^{(3)} ,m+6^{(3)},\hdots, n+m-1^{(3)},\mathbf{n+m+1},n+m-1^{(3)},n+m-3^{(3)},\hdots, n+3^{(3)},n+1^{(2)}.\]
    If $n$ and $m$ are odd, the sequence of the size of the channels is
    \[m+3^{(3)},m+5^{(3)},\hdots,n+m,\mathbf{n+m},n+m, n+m-2^{(3)},\hdots, n+1^{(3)}.\]

\end{proposition}
The proof of this proposition follows immediately from Proposition \ref{convex} and Lemma \ref{convexlemma}, using an inductive argument.

This notation is useful as it immediately allows us to calculate the size of the shadow. Summing up the sequences above, we get the following corollary.

\begin{corollary}
    Let $S_{n,m}$ be the size of the shadow of the alcove $(n,m)$, with $n,m\geq 2$. Then
    \[S_{n,m}=\begin{cases}
        \dfrac{12nm+3n^2+3m^2-6n-6m}{4}& \text{ if $n$ even and $m$ even,}\\
        \dfrac{12nm+3n^2+3m^2-6n-12m+9}{4}& \text{ if $n$ even and $m$ odd,}\\
        \dfrac{12nm+3n^2+3m^2-6m-11}{4}& \text{ if $n$ odd and $m$ even,}\\
        \dfrac{12nm+3n^2+3m^2-12m-6}{4}& \text{ if $n$ odd and $m$ odd.}
            
    \end{cases}\]
\end{corollary}

We can also conclude from Proposition \ref{numchannels} the following proposition, which tells us how many elements are in each channel within the shadow. 

\begin{proposition}
    Let $m,n\in\mathbb{Z}$ such that $n\geq 2$ and $m\geq 0$. If $m=0$, and $4-2n\leq 2c\leq  n-\epsilon_n$, let 
    \[k(n,m,c)=
2\left\lfloor
\dfrac{n+c+1+\epsilon_n}{3}
\right\rfloor.
\]
    
If $m\neq 0$, $4-2n-m+\epsilon_m
\leq 2c \leq
n-m-2\epsilon_n+3\epsilon_m
-\epsilon_n\epsilon_m$, let \[
k(n,m,c)=
2\left\lfloor
\dfrac{
2n+4m+2c-2+6\epsilon_n
}{6}
\right\rfloor,\]

and if $m\neq 0$, $n-m+2-3\epsilon_n+3\epsilon_m
\leq 2c \leq
n+2m-3\epsilon_n,$ let
\[k(n,m,c)=
2\left\lfloor
\dfrac{
4n+2m-2c+2-2\epsilon_m
}{6}
\right\rfloor.\]
Otherwise, let $k(n,m,c)=0$. Then the $(\alpha+\beta)$-channel containing $(0,c)$ has  $k(n,m,c)$ elements in the shadow of $(n,m)$.

\end{proposition}
Finally, we conclude the following theorem and proposition, which gives an efficient process to compute the shadow of an element. Note that $(i,j)$ lies in the channel containing $(0,(i+2j)/2)$, by Proposition \ref{alphaplusbeta}. So there are $k(n,m,(i+2j)/2)$ elements in the shadow of $(n,m)$ lying in the channel containing $(i,j)$. Then the theorem follows from the symmetry of the shadow, established in Lemma \ref{symmetry}, noting that $i$ is specified to be even.
\begin{theorem}\label{shadowscondition}



Let $(i,j)$ be an alcove written in even-coordinate notation. Then $(i,j)$ is in the shadow of $(n,m)$ if and only if 
    \[2-k(n,m,({i+2j})/{2})\leq i \leq k(n,m,({i+2j})/{2}).\]

\end{theorem}



\begin{corollary}
For reduced words $x,y$, it is possible to determine whether $x\leq y$ in $O(\ell(y))$ time. 
\end{corollary}
\begin{proof}
    If $\ell(x)>\ell(y)$, then $x\not\leq y$. Otherwise, we use the \hyperref[algorithm]{Words-to-Coordinates Algorithm} on both $x$ and $y$, which computes the coordinates of the corresponding alcove in linear time with respect to their lengths. Therefore, this has complexity $O(\ell(y))$. Then Theorem \ref{shadowscondition} determines whether the second alcove lies in the shadow of the first alcove using two inequalities requiring only a fixed number of arithmetic operations and floor functions. Consequently, deciding whether one alcove lies in the shadow of another has linear time complexity with respect to the lengths of the words, with the shadow-membership test itself requiring constant time once the coordinates are known.
\end{proof}

\section*{Acknowledgements}
I would like to express my gratitude to both my supervisors, Yusra Naqvi and Beth Romano, for their support throughout my PhD, and their feedback on drafts of this paper. The software written by Gibson \cite{GIBSON} has been used to generate exemplifying pictures in this paper. This work was supported by the Engineering and Physical Sciences
Research Council [EP/S021590/1], the EPSRC Centre for Doctoral Training in Geometry and
Number Theory (The London School of Geometry and Number Theory), University College
London.


\bibliographystyle{acm}

\bibliography{references}

\bigskip

\small\textsc{Megan Masters} \texttt{megan.masters.22@ucl.ac.uk}

\small\textsc{Department of Mathematics,
University College London,
 WC1E 6BT, UK}

\end{document}

%% file: Pictures/use616241.eps_tex
\begingroup%
  \makeatletter%
  \providecommand\color[2][]{%
    \errmessage{(Inkscape) Color is used for the text in Inkscape, but the package 'color.sty' is not loaded}%
    \renewcommand\color[2][]{}%
  }%
  \providecommand\transparent[1]{%
    \errmessage{(Inkscape) Transparency is used (non-zero) for the text in Inkscape, but the package 'transparent.sty' is not loaded}%
    \renewcommand\transparent[1]{}%
  }%
  \providecommand\rotatebox[2]{#2}%
  \newcommand*\fsize{\dimexpr\f@size pt\relax}%
  \newcommand*\lineheight[1]{\fontsize{\fsize}{#1\fsize}\selectfont}%
  \ifx\svgwidth\undefined%
    \setlength{\unitlength}{294.58441787bp}%
    \ifx\svgscale\undefined%
      \relax%
    \else%
      \setlength{\unitlength}{\unitlength * \real{\svgscale}}%
    \fi%
  \else%
    \setlength{\unitlength}{\svgwidth}%
  \fi%
  \global\let\svgwidth\undefined%
  \global\let\svgscale\undefined%
  \makeatother%
  \begin{picture}(1,0.89062491)%
    \lineheight{1}%
    \setlength\tabcolsep{0pt}%
    \put(0,0){\includegraphics[width=\unitlength]{use616241.eps}}%
    \put(0.38867161,0.49753666){\color[rgb]{0,0,0}\makebox(0,0)[lt]{\lineheight{1.25}\smash{\begin{tabular}[t]{l}$\alpha+\beta$\end{tabular}}}}%
    \put(0.7009248,0.35639106){\color[rgb]{0,0,0}\makebox(0,0)[lt]{\lineheight{1.25}\smash{\begin{tabular}[t]{l}$\alpha$\end{tabular}}}}%
    \put(0.21349342,0.37779084){\color[rgb]{0,0,0}\makebox(0,0)[lt]{\lineheight{1.25}\smash{\begin{tabular}[t]{l}$\beta$\end{tabular}}}}%
  \end{picture}%
\endgroup%

%% file: Pictures/use614592.eps_tex
\begingroup%
  \makeatletter%
  \providecommand\color[2][]{%
    \errmessage{(Inkscape) Color is used for the text in Inkscape, but the package 'color.sty' is not loaded}%
    \renewcommand\color[2][]{}%
  }%
  \providecommand\transparent[1]{%
    \errmessage{(Inkscape) Transparency is used (non-zero) for the text in Inkscape, but the package 'transparent.sty' is not loaded}%
    \renewcommand\transparent[1]{}%
  }%
  \providecommand\rotatebox[2]{#2}%
  \newcommand*\fsize{\dimexpr\f@size pt\relax}%
  \newcommand*\lineheight[1]{\fontsize{\fsize}{#1\fsize}\selectfont}%
  \ifx\svgwidth\undefined%
    \setlength{\unitlength}{381.34570024bp}%
    \ifx\svgscale\undefined%
      \relax%
    \else%
      \setlength{\unitlength}{\unitlength * \real{\svgscale}}%
    \fi%
  \else%
    \setlength{\unitlength}{\svgwidth}%
  \fi%
  \global\let\svgwidth\undefined%
  \global\let\svgscale\undefined%
  \makeatother%
  \begin{picture}(1,0.78662876)%
    \lineheight{1}%
    \setlength\tabcolsep{0pt}%
    \put(0,0){\includegraphics[width=\unitlength]{use614592.eps}}%
    \put(0.38702521,0.21739643){\color[rgb]{0,0,0}\makebox(0,0)[lt]{\lineheight{1.25}\smash{\begin{tabular}[t]{l}$n$ negative\end{tabular}}}}%
    \put(0.55843346,0.52465104){\color[rgb]{0,0,0}\makebox(0,0)[lt]{\lineheight{1.25}\smash{\begin{tabular}[t]{l}$n$ positive\end{tabular}}}}%
  \end{picture}%
\endgroup%

%% file: Pictures/use61459.eps_tex
\begingroup%
  \makeatletter%
  \providecommand\color[2][]{%
    \errmessage{(Inkscape) Color is used for the text in Inkscape, but the package 'color.sty' is not loaded}%
    \renewcommand\color[2][]{}%
  }%
  \providecommand\transparent[1]{%
    \errmessage{(Inkscape) Transparency is used (non-zero) for the text in Inkscape, but the package 'transparent.sty' is not loaded}%
    \renewcommand\transparent[1]{}%
  }%
  \providecommand\rotatebox[2]{#2}%
  \newcommand*\fsize{\dimexpr\f@size pt\relax}%
  \newcommand*\lineheight[1]{\fontsize{\fsize}{#1\fsize}\selectfont}%
  \ifx\svgwidth\undefined%
    \setlength{\unitlength}{381.34570024bp}%
    \ifx\svgscale\undefined%
      \relax%
    \else%
      \setlength{\unitlength}{\unitlength * \real{\svgscale}}%
    \fi%
  \else%
    \setlength{\unitlength}{\svgwidth}%
  \fi%
  \global\let\svgwidth\undefined%
  \global\let\svgscale\undefined%
  \makeatother%
  \begin{picture}(1,0.78662876)%
    \lineheight{1}%
    \setlength\tabcolsep{0pt}%
    \put(0,0){\includegraphics[width=\unitlength]{use61459.eps}}%
    \put(0.15220645,0.68970519){\color[rgb]{0,0,0}\makebox(0,0)[lt]{\lineheight{1.25}\smash{\begin{tabular}[t]{l}$m$ negative\end{tabular}}}}%
    \put(0.01789294,0.06163541){\color[rgb]{0,0,0}\makebox(0,0)[lt]{\lineheight{1.25}\smash{\begin{tabular}[t]{l}$m$ negative\end{tabular}}}}%
    \put(0.67365394,0.06163541){\color[rgb]{0,0,0}\makebox(0,0)[lt]{\lineheight{1.25}\smash{\begin{tabular}[t]{l}$m$ positive\end{tabular}}}}%
    \put(0.69214612,0.68894691){\color[rgb]{0,0,0}\makebox(0,0)[lt]{\lineheight{1.25}\smash{\begin{tabular}[t]{l}$m$ positive\end{tabular}}}}%
  \end{picture}%
\endgroup%

%% file: Pictures/use62127.eps_tex
\begingroup%
  \makeatletter%
  \providecommand\color[2][]{%
    \errmessage{(Inkscape) Color is used for the text in Inkscape, but the package 'color.sty' is not loaded}%
    \renewcommand\color[2][]{}%
  }%
  \providecommand\transparent[1]{%
    \errmessage{(Inkscape) Transparency is used (non-zero) for the text in Inkscape, but the package 'transparent.sty' is not loaded}%
    \renewcommand\transparent[1]{}%
  }%
  \providecommand\rotatebox[2]{#2}%
  \newcommand*\fsize{\dimexpr\f@size pt\relax}%
  \newcommand*\lineheight[1]{\fontsize{\fsize}{#1\fsize}\selectfont}%
  \ifx\svgwidth\undefined%
    \setlength{\unitlength}{367.4633655bp}%
    \ifx\svgscale\undefined%
      \relax%
    \else%
      \setlength{\unitlength}{\unitlength * \real{\svgscale}}%
    \fi%
  \else%
    \setlength{\unitlength}{\svgwidth}%
  \fi%
  \global\let\svgwidth\undefined%
  \global\let\svgscale\undefined%
  \makeatother%
  \begin{picture}(1,0.74530269)%
    \lineheight{1}%
    \setlength\tabcolsep{0pt}%
    \put(0,0){\includegraphics[width=\unitlength]{use62127.eps}}%
    \put(0.45093853,0.61283133){\color[rgb]{0,0,0}\makebox(0,0)[lt]{\lineheight{1.25}\smash{\begin{tabular}[t]{l}$s_0$-region\end{tabular}}}}%
    \put(0.57649708,0.15229129){\color[rgb]{0,0,0}\makebox(0,0)[lt]{\lineheight{1.25}\smash{\begin{tabular}[t]{l}$s_1$-region\end{tabular}}}}%
    \put(0.06762444,0.31982782){\color[rgb]{0,0,0}\makebox(0,0)[lt]{\lineheight{1.25}\smash{\begin{tabular}[t]{l}$s_2$-region\end{tabular}}}}%
  \end{picture}%
\endgroup%

%% file: Pictures/use61376.eps_tex
\begingroup%
  \makeatletter%
  \providecommand\color[2][]{%
    \errmessage{(Inkscape) Color is used for the text in Inkscape, but the package 'color.sty' is not loaded}%
    \renewcommand\color[2][]{}%
  }%
  \providecommand\transparent[1]{%
    \errmessage{(Inkscape) Transparency is used (non-zero) for the text in Inkscape, but the package 'transparent.sty' is not loaded}%
    \renewcommand\transparent[1]{}%
  }%
  \providecommand\rotatebox[2]{#2}%
  \newcommand*\fsize{\dimexpr\f@size pt\relax}%
  \newcommand*\lineheight[1]{\fontsize{\fsize}{#1\fsize}\selectfont}%
  \ifx\svgwidth\undefined%
    \setlength{\unitlength}{326.80457168bp}%
    \ifx\svgscale\undefined%
      \relax%
    \else%
      \setlength{\unitlength}{\unitlength * \real{\svgscale}}%
    \fi%
  \else%
    \setlength{\unitlength}{\svgwidth}%
  \fi%
  \global\let\svgwidth\undefined%
  \global\let\svgscale\undefined%
  \makeatother%
  \begin{picture}(1,0.79107981)%
    \lineheight{1}%
    \setlength\tabcolsep{0pt}%
    \put(0,0){\includegraphics[width=\unitlength]{use61376.eps}}%
    \put(0.24647923,0.5136096){\color[rgb]{0,0,0}\makebox(0,0)[lt]{\lineheight{1.25}\smash{\begin{tabular}[t]{l}$\dfrac{x+1}{2}$\end{tabular}}}}%
    \put(0.3705125,0.18218026){\color[rgb]{0,0,0}\makebox(0,0)[lt]{\lineheight{1.25}\smash{\begin{tabular}[t]{l}$\dfrac{n-2}{2}$\end{tabular}}}}%
    \put(0.47262518,0.25495021){\color[rgb]{0,0,0}\makebox(0,0)[lt]{\lineheight{1.25}\smash{\begin{tabular}[t]{l}$\dfrac{m}{2}$\end{tabular}}}}%
    \put(0.6475078,0.56950419){\color[rgb]{0,0,0}\makebox(0,0)[lt]{\lineheight{1.25}\smash{\begin{tabular}[t]{l}$\dfrac{y-2}{2}$\end{tabular}}}}%
  \end{picture}%
\endgroup%

%% file: Pictures/use615372.eps_tex
\begingroup%
  \makeatletter%
  \providecommand\color[2][]{%
    \errmessage{(Inkscape) Color is used for the text in Inkscape, but the package 'color.sty' is not loaded}%
    \renewcommand\color[2][]{}%
  }%
  \providecommand\transparent[1]{%
    \errmessage{(Inkscape) Transparency is used (non-zero) for the text in Inkscape, but the package 'transparent.sty' is not loaded}%
    \renewcommand\transparent[1]{}%
  }%
  \providecommand\rotatebox[2]{#2}%
  \newcommand*\fsize{\dimexpr\f@size pt\relax}%
  \newcommand*\lineheight[1]{\fontsize{\fsize}{#1\fsize}\selectfont}%
  \ifx\svgwidth\undefined%
    \setlength{\unitlength}{276.10949204bp}%
    \ifx\svgscale\undefined%
      \relax%
    \else%
      \setlength{\unitlength}{\unitlength * \real{\svgscale}}%
    \fi%
  \else%
    \setlength{\unitlength}{\svgwidth}%
  \fi%
  \global\let\svgwidth\undefined%
  \global\let\svgscale\undefined%
  \makeatother%
  \begin{picture}(1,0.67387034)%
    \lineheight{1}%
    \setlength\tabcolsep{0pt}%
    \put(0,0){\includegraphics[width=\unitlength]{use615372.eps}}%
    \put(0.53704544,0.26881018){\color[rgb]{0,0.01176471,0.07058824}\makebox(0,0)[lt]{\lineheight{1.25}\smash{\begin{tabular}[t]{l}$\dfrac{x-1}{2}$\end{tabular}}}}%
    \put(0.51009807,0.54751631){\color[rgb]{0,0.01176471,0.07058824}\makebox(0,0)[lt]{\lineheight{1.25}\smash{\begin{tabular}[t]{l}$\dfrac{y+2}{2}$\end{tabular}}}}%
  \end{picture}%
\endgroup%

%% file: Pictures/use62128.eps_tex
\begingroup%
  \makeatletter%
  \providecommand\color[2][]{%
    \errmessage{(Inkscape) Color is used for the text in Inkscape, but the package 'color.sty' is not loaded}%
    \renewcommand\color[2][]{}%
  }%
  \providecommand\transparent[1]{%
    \errmessage{(Inkscape) Transparency is used (non-zero) for the text in Inkscape, but the package 'transparent.sty' is not loaded}%
    \renewcommand\transparent[1]{}%
  }%
  \providecommand\rotatebox[2]{#2}%
  \newcommand*\fsize{\dimexpr\f@size pt\relax}%
  \newcommand*\lineheight[1]{\fontsize{\fsize}{#1\fsize}\selectfont}%
  \ifx\svgwidth\undefined%
    \setlength{\unitlength}{437.76100806bp}%
    \ifx\svgscale\undefined%
      \relax%
    \else%
      \setlength{\unitlength}{\unitlength * \real{\svgscale}}%
    \fi%
  \else%
    \setlength{\unitlength}{\svgwidth}%
  \fi%
  \global\let\svgwidth\undefined%
  \global\let\svgscale\undefined%
  \makeatother%
  \begin{picture}(1,0.66976452)%
    \lineheight{1}%
    \setlength\tabcolsep{0pt}%
    \put(0,0){\includegraphics[width=\unitlength]{use62128.eps}}%
    \put(0.15325342,0.19743741){\color[rgb]{0,0,0}\makebox(0,0)[lt]{\lineheight{1.25}\smash{\begin{tabular}[t]{l}$(0,0)$\\$j= 0 $\end{tabular}}}}%
    \put(0.22889658,0.15856052){\color[rgb]{0,0,0}\makebox(0,0)[lt]{\lineheight{1.25}\smash{\begin{tabular}[t]{l}$(0,1)$\\$j= 1 $\end{tabular}}}}%
    \put(0.30535721,0.19697371){\color[rgb]{0,0,0}\makebox(0,0)[lt]{\lineheight{1.25}\smash{\begin{tabular}[t]{l}$(0,2)$\\$j= 2 $\end{tabular}}}}%
    \put(0.37872675,0.16022461){\color[rgb]{0,0,0}\makebox(0,0)[lt]{\lineheight{1.25}\smash{\begin{tabular}[t]{l}$(0,3)$\\$j= 3 $\end{tabular}}}}%
    \put(0.4401398,0.19386814){\color[rgb]{0,0,0}\makebox(0,0)[lt]{\lineheight{1.25}\smash{\begin{tabular}[t]{l}$(0,4)$\\$j= 4 $\end{tabular}}}}%
    \put(0.52057457,0.16177739){\color[rgb]{0,0,0}\makebox(0,0)[lt]{\lineheight{1.25}\smash{\begin{tabular}[t]{l}$(0,5)$\\$j= 5 $\end{tabular}}}}%
    \put(0.22872678,0.3237838){\color[rgb]{0,0,0}\makebox(0,0)[lt]{\lineheight{1.25}\smash{\begin{tabular}[t]{l}$(2,0)$\\$j= 4 $\end{tabular}}}}%
    \put(0.30372676,0.27875327){\color[rgb]{0,0,0}\makebox(0,0)[lt]{\lineheight{1.25}\smash{\begin{tabular}[t]{l}$(2,1)$\\$j= 5 $\end{tabular}}}}%
    \put(0.37655285,0.32585422){\color[rgb]{0,0,0}\makebox(0,0)[lt]{\lineheight{1.25}\smash{\begin{tabular}[t]{l}$(2,2)$\\$j= 0 $\end{tabular}}}}%
    \put(0.45644415,0.27771812){\color[rgb]{0,0,0}\makebox(0,0)[lt]{\lineheight{1.25}\smash{\begin{tabular}[t]{l}$(2,3)$\\$j= 1 $\end{tabular}}}}%
    \put(0.53307458,0.32999493){\color[rgb]{0,0,0}\makebox(0,0)[lt]{\lineheight{1.25}\smash{\begin{tabular}[t]{l}$(2,4)$\\$j= 2 $\end{tabular}}}}%
    \put(0.60427023,0.2782357){\color[rgb]{0,0,0}\makebox(0,0)[lt]{\lineheight{1.25}\smash{\begin{tabular}[t]{l}$(2,5)$\\$j= 3 $\end{tabular}}}}%
    \put(0.30209634,0.43661896){\color[rgb]{0,0,0}\makebox(0,0)[lt]{\lineheight{1.25}\smash{\begin{tabular}[t]{l}$(4,0)$\\$j= 2 $\end{tabular}}}}%
    \put(0.38144415,0.40245785){\color[rgb]{0,0,0}\makebox(0,0)[lt]{\lineheight{1.25}\smash{\begin{tabular}[t]{l}$(4,1)$\\$j= 3 $\end{tabular}}}}%
    \put(0.45046588,0.4469708){\color[rgb]{0,0,0}\makebox(0,0)[lt]{\lineheight{1.25}\smash{\begin{tabular}[t]{l}$(4,2)$\\$j= 4 $\end{tabular}}}}%
    \put(0.52818327,0.40349306){\color[rgb]{0,0,0}\makebox(0,0)[lt]{\lineheight{1.25}\smash{\begin{tabular}[t]{l}$(4,3)$\\$j= 5 $\end{tabular}}}}%
    \put(0.6048137,0.4402421){\color[rgb]{0,0,0}\makebox(0,0)[lt]{\lineheight{1.25}\smash{\begin{tabular}[t]{l}$(4,4)$\\$j= 0 $\end{tabular}}}}%
    \put(0.68307456,0.40659863){\color[rgb]{0,0,0}\makebox(0,0)[lt]{\lineheight{1.25}\smash{\begin{tabular}[t]{l}$(4,5)$\\$j= 1 $\end{tabular}}}}%
  \end{picture}%
\endgroup%

%% file: Pictures/use62045.eps_tex
\begingroup%
  \makeatletter%
  \providecommand\color[2][]{%
    \errmessage{(Inkscape) Color is used for the text in Inkscape, but the package 'color.sty' is not loaded}%
    \renewcommand\color[2][]{}%
  }%
  \providecommand\transparent[1]{%
    \errmessage{(Inkscape) Transparency is used (non-zero) for the text in Inkscape, but the package 'transparent.sty' is not loaded}%
    \renewcommand\transparent[1]{}%
  }%
  \providecommand\rotatebox[2]{#2}%
  \newcommand*\fsize{\dimexpr\f@size pt\relax}%
  \newcommand*\lineheight[1]{\fontsize{\fsize}{#1\fsize}\selectfont}%
  \ifx\svgwidth\undefined%
    \setlength{\unitlength}{463.25639794bp}%
    \ifx\svgscale\undefined%
      \relax%
    \else%
      \setlength{\unitlength}{\unitlength * \real{\svgscale}}%
    \fi%
  \else%
    \setlength{\unitlength}{\svgwidth}%
  \fi%
  \global\let\svgwidth\undefined%
  \global\let\svgscale\undefined%
  \makeatother%
  \begin{picture}(1,0.65105389)%
    \lineheight{1}%
    \setlength\tabcolsep{0pt}%
    \put(0,0){\includegraphics[width=\unitlength]{use62045.eps}}%
    \put(0.36110292,0.25854028){\color[rgb]{0,0,0}\makebox(0,0)[lt]{\lineheight{1.25}\smash{\begin{tabular}[t]{l}$\beta$\end{tabular}}}}%
    \put(0.46973372,0.32973667){\color[rgb]{0,0,0}\makebox(0,0)[lt]{\lineheight{1.25}\smash{\begin{tabular}[t]{l}$\alpha+\beta$\end{tabular}}}}%
    \put(0.62320412,0.25746292){\color[rgb]{0,0,0}\makebox(0,0)[lt]{\lineheight{1.25}\smash{\begin{tabular}[t]{l}$\alpha$\end{tabular}}}}%
  \end{picture}%
\endgroup%

%% file: Pictures/use62046.eps_tex
\begingroup%
  \makeatletter%
  \providecommand\color[2][]{%
    \errmessage{(Inkscape) Color is used for the text in Inkscape, but the package 'color.sty' is not loaded}%
    \renewcommand\color[2][]{}%
  }%
  \providecommand\transparent[1]{%
    \errmessage{(Inkscape) Transparency is used (non-zero) for the text in Inkscape, but the package 'transparent.sty' is not loaded}%
    \renewcommand\transparent[1]{}%
  }%
  \providecommand\rotatebox[2]{#2}%
  \newcommand*\fsize{\dimexpr\f@size pt\relax}%
  \newcommand*\lineheight[1]{\fontsize{\fsize}{#1\fsize}\selectfont}%
  \ifx\svgwidth\undefined%
    \setlength{\unitlength}{460.28809969bp}%
    \ifx\svgscale\undefined%
      \relax%
    \else%
      \setlength{\unitlength}{\unitlength * \real{\svgscale}}%
    \fi%
  \else%
    \setlength{\unitlength}{\svgwidth}%
  \fi%
  \global\let\svgwidth\undefined%
  \global\let\svgscale\undefined%
  \makeatother%
  \begin{picture}(1,0.54833336)%
    \lineheight{1}%
    \setlength\tabcolsep{0pt}%
    \put(0,0){\includegraphics[width=\unitlength]{use62046.eps}}%
    \put(0.36189318,0.21452481){\color[rgb]{0,0,0}\makebox(0,0)[lt]{\lineheight{1.25}\smash{\begin{tabular}[t]{l}$\beta$\end{tabular}}}}%
    \put(0.46715956,0.28432311){\color[rgb]{0,0,0}\makebox(0,0)[lt]{\lineheight{1.25}\smash{\begin{tabular}[t]{l}$\alpha+\beta$\end{tabular}}}}%
    \put(0.62161965,0.21158328){\color[rgb]{0,0,0}\makebox(0,0)[lt]{\lineheight{1.25}\smash{\begin{tabular}[t]{l}$\alpha$\end{tabular}}}}%
  \end{picture}%
\endgroup%

%% file: Pictures/use62047.eps_tex
\begingroup%
  \makeatletter%
  \providecommand\color[2][]{%
    \errmessage{(Inkscape) Color is used for the text in Inkscape, but the package 'color.sty' is not loaded}%
    \renewcommand\color[2][]{}%
  }%
  \providecommand\transparent[1]{%
    \errmessage{(Inkscape) Transparency is used (non-zero) for the text in Inkscape, but the package 'transparent.sty' is not loaded}%
    \renewcommand\transparent[1]{}%
  }%
  \providecommand\rotatebox[2]{#2}%
  \newcommand*\fsize{\dimexpr\f@size pt\relax}%
  \newcommand*\lineheight[1]{\fontsize{\fsize}{#1\fsize}\selectfont}%
  \ifx\svgwidth\undefined%
    \setlength{\unitlength}{460.28809969bp}%
    \ifx\svgscale\undefined%
      \relax%
    \else%
      \setlength{\unitlength}{\unitlength * \real{\svgscale}}%
    \fi%
  \else%
    \setlength{\unitlength}{\svgwidth}%
  \fi%
  \global\let\svgwidth\undefined%
  \global\let\svgscale\undefined%
  \makeatother%
  \begin{picture}(1,0.54833336)%
    \lineheight{1}%
    \setlength\tabcolsep{0pt}%
    \put(0,0){\includegraphics[width=\unitlength]{use62047.eps}}%
    \put(0.48947872,0.16976518){\color[rgb]{0,0,0}\makebox(0,0)[lt]{\lineheight{1.25}\smash{\begin{tabular}[t]{l}$1$\end{tabular}}}}%
    \put(0.50645884,0.20840972){\color[rgb]{0,0,0}\makebox(0,0)[lt]{\lineheight{1.25}\smash{\begin{tabular}[t]{l}$s_0$\end{tabular}}}}%
    \put(0.50831651,0.348575){\color[rgb]{0,0,0}\makebox(0,0)[lt]{\lineheight{1.25}\smash{\begin{tabular}[t]{l}$s_0$\end{tabular}}}}%
    \put(0.50644151,0.48920001){\color[rgb]{0,0,0}\makebox(0,0)[lt]{\lineheight{1.25}\smash{\begin{tabular}[t]{l}$s_0$\end{tabular}}}}%
    \put(0.50748318,0.04815832){\color[rgb]{0,0,0}\makebox(0,0)[lt]{\lineheight{1.25}\smash{\begin{tabular}[t]{l}$s_0$\end{tabular}}}}%
    \put(0.52504301,0.2815214){\color[rgb]{0,0,0}\makebox(0,0)[lt]{\lineheight{1.25}\smash{\begin{tabular}[t]{l}$s_1s_2s_1$\end{tabular}}}}%
    \put(0.52451013,0.42042868){\color[rgb]{0,0,0}\makebox(0,0)[lt]{\lineheight{1.25}\smash{\begin{tabular}[t]{l}$s_1s_2s_1$\end{tabular}}}}%
    \put(0.52451013,0.12815787){\color[rgb]{0,0,0}\makebox(0,0)[lt]{\lineheight{1.25}\smash{\begin{tabular}[t]{l}$s_1s_2s_1$\end{tabular}}}}%
  \end{picture}%
\endgroup%

%% file: Pictures/use61457.eps_tex
\begingroup%
  \makeatletter%
  \providecommand\color[2][]{%
    \errmessage{(Inkscape) Color is used for the text in Inkscape, but the package 'color.sty' is not loaded}%
    \renewcommand\color[2][]{}%
  }%
  \providecommand\transparent[1]{%
    \errmessage{(Inkscape) Transparency is used (non-zero) for the text in Inkscape, but the package 'transparent.sty' is not loaded}%
    \renewcommand\transparent[1]{}%
  }%
  \providecommand\rotatebox[2]{#2}%
  \newcommand*\fsize{\dimexpr\f@size pt\relax}%
  \newcommand*\lineheight[1]{\fontsize{\fsize}{#1\fsize}\selectfont}%
  \ifx\svgwidth\undefined%
    \setlength{\unitlength}{346.7504234bp}%
    \ifx\svgscale\undefined%
      \relax%
    \else%
      \setlength{\unitlength}{\unitlength * \real{\svgscale}}%
    \fi%
  \else%
    \setlength{\unitlength}{\svgwidth}%
  \fi%
  \global\let\svgwidth\undefined%
  \global\let\svgscale\undefined%
  \makeatother%
  \begin{picture}(1,0.90044228)%
    \lineheight{1}%
    \setlength\tabcolsep{0pt}%
    \put(0,0){\includegraphics[width=\unitlength]{use61457.eps}}%
  \end{picture}%
\endgroup%

%% file: Pictures/use61458.eps_tex
\begingroup%
  \makeatletter%
  \providecommand\color[2][]{%
    \errmessage{(Inkscape) Color is used for the text in Inkscape, but the package 'color.sty' is not loaded}%
    \renewcommand\color[2][]{}%
  }%
  \providecommand\transparent[1]{%
    \errmessage{(Inkscape) Transparency is used (non-zero) for the text in Inkscape, but the package 'transparent.sty' is not loaded}%
    \renewcommand\transparent[1]{}%
  }%
  \providecommand\rotatebox[2]{#2}%
  \newcommand*\fsize{\dimexpr\f@size pt\relax}%
  \newcommand*\lineheight[1]{\fontsize{\fsize}{#1\fsize}\selectfont}%
  \ifx\svgwidth\undefined%
    \setlength{\unitlength}{346.7504234bp}%
    \ifx\svgscale\undefined%
      \relax%
    \else%
      \setlength{\unitlength}{\unitlength * \real{\svgscale}}%
    \fi%
  \else%
    \setlength{\unitlength}{\svgwidth}%
  \fi%
  \global\let\svgwidth\undefined%
  \global\let\svgscale\undefined%
  \makeatother%
  \begin{picture}(1,0.90044228)%
    \lineheight{1}%
    \setlength\tabcolsep{0pt}%
    \put(0,0){\includegraphics[width=\unitlength]{use61458.eps}}%
  \end{picture}%
\endgroup%

%% file: Pictures/use50311.eps_tex
\begingroup%
  \makeatletter%
  \providecommand\color[2][]{%
    \errmessage{(Inkscape) Color is used for the text in Inkscape, but the package 'color.sty' is not loaded}%
    \renewcommand\color[2][]{}%
  }%
  \providecommand\transparent[1]{%
    \errmessage{(Inkscape) Transparency is used (non-zero) for the text in Inkscape, but the package 'transparent.sty' is not loaded}%
    \renewcommand\transparent[1]{}%
  }%
  \providecommand\rotatebox[2]{#2}%
  \newcommand*\fsize{\dimexpr\f@size pt\relax}%
  \newcommand*\lineheight[1]{\fontsize{\fsize}{#1\fsize}\selectfont}%
  \ifx\svgwidth\undefined%
    \setlength{\unitlength}{167.62160894bp}%
    \ifx\svgscale\undefined%
      \relax%
    \else%
      \setlength{\unitlength}{\unitlength * \real{\svgscale}}%
    \fi%
  \else%
    \setlength{\unitlength}{\svgwidth}%
  \fi%
  \global\let\svgwidth\undefined%
  \global\let\svgscale\undefined%
  \makeatother%
  \begin{picture}(1,0.91990837)%
    \lineheight{1}%
    \setlength\tabcolsep{0pt}%
    \put(0,0){\includegraphics[width=\unitlength]{use50311.eps}}%
    \put(0.35279045,0.53394075){\color[rgb]{0.4,0.4,0.4}\makebox(0,0)[lt]{\lineheight{1.25}\smash{\begin{tabular}[t]{l}$(1,0)$\end{tabular}}}}%
    \put(0.47075482,0.33402522){\color[rgb]{0.4,0.4,0.4}\makebox(0,0)[lt]{\lineheight{1.25}\smash{\begin{tabular}[t]{l}$(0,1)$\end{tabular}}}}%
    \put(0.35806916,0.41278025){\color[rgb]{0.4,0.4,0.4}\makebox(0,0)[lt]{\lineheight{1.25}\smash{\begin{tabular}[t]{l}$(0,0)$\end{tabular}}}}%
    \put(0.48169239,0.58030661){\color[rgb]{0.4,0.4,0.4}\makebox(0,0)[lt]{\lineheight{1.25}\smash{\begin{tabular}[t]{l}$(2,0)$\end{tabular}}}}%
  \end{picture}%
\endgroup%